\documentclass{article}
\usepackage{graphicx} 
\usepackage[margin=1in]{geometry}
\usepackage{amsfonts, amsmath, amssymb, amscd, amsthm, color, graphicx, mathabx, mathrsfs, wasysym, setspace, mdwlist, calc, float, mathtools, tikz-cd}
\usepackage{bbm}
\usepackage[noadjust]{cite}
\usepackage{hyperref}
\usepackage{authblk}
\usepackage{mathdots}
\usepackage{enumerate}
\usepackage{comment}
\usepackage{stmaryrd}
\usepackage[normalem]{ulem}
\usepackage[toc,page]{appendix}

\newtheorem{thm}{Theorem}[section]
\newtheorem{main}{Theorem}

\newtheorem{thma}[main]{Theorem}
\newtheorem{prop}[thm]{Proposition}

\newtheorem{cor}[thm]{Corollary}

\newtheorem{claim}[thm]{Claim}

\newcounter{openproblem}
\renewcommand{\theopenproblem}{\Alph{openproblem}}

\newcommand{\sR}{\mathscr R}
\newcommand{\sN}{\mathscr N}

\renewcommand{\phi}{\varphi}
\newcommand{\WR}{\mathcal{WR}}

\renewcommand{\P }{\mathcal P}
\newcommand{\M}{\mathcal M}
\newcommand{\Nn}{\mathcal N}
\newcommand{\ra}{\rightarrow}
\newcommand{\ca}{\curvearrowright}
\newcommand{\A}{\mathcal A}
\newcommand{\B}{\mathcal B}

\renewcommand{\L}{\mathcal L}
\newcommand{\Q}{\mathcal Q}
\newcommand{\R}{\mathcal R}
\newcommand{\T}{\mathcal T}

\newcommand{\W}{\mathcal W}

\newcommand{\sU}{\mathscr U}

\newcommand{\sZ}{\mathscr Z}

\theoremstyle{remark}

\theoremstyle{definition}

\newtheorem{defn}[thm]{Definition}

\newtheorem{rmk}[thm]{Remark}

\DeclareMathOperator{\tr}{tr}
\DeclareMathOperator{\spa}{span}

\DeclareMathOperator{\Out}{Out}
\DeclareMathOperator{\Aut}{Aut}
\DeclareMathOperator{\Inn}{Inn}
\DeclareMathOperator{\Ad}{ \rm ad}
\DeclareMathOperator{\Id}{\rm id}

\numberwithin{equation}{section}

\title{Property (T) group factors whose Jones index set equals all positive integers}
\author{Ionu\c t Chifan and Junhwi Lim}
\date{}
\AtEndDocument{%
  \par
  \noindent
  \medskip
  \begin{tabular}{@{}l@{}}%
    \textsc{Department of Mathematics, The University of Iowa, Iowa City, IA 52242, USA}\\
    \textit{E-mail address}: \texttt{ionut-chifan@uiowa.edu}
  \end{tabular}

  \par
  \noindent
  \medskip
  \begin{tabular}{@{}l@{}}%
    \textsc{Department of Mathematics, Vanderbilt University, Nashville, TN 37240, USA}\\
    \textit{E-mail address}: \texttt{junhwi.lim@vanderbilt.edu}
  \end{tabular}
  }
\begin{document}

\maketitle

\begin{abstract}Using a mélange of techniques at the rich intersection of deformation/rigidity theory, finite index subfactor theory, and geometric group theory, we prove the existence of a continuum of property~(T) factors that are pairwise non–stably isomorphic and whose Jones index sets consist of all positive integers. These factors are realized as group von Neumann algebras $\L(G)$ associated with property~(T) generalized wreath-like product groups $G \in \mathscr{WR}(A, B \curvearrowright I)$ introduced in~\cite{CIOS23b}, where $A$ is abelian, $B$ is a non-parabolic subgroup of a relatively hyperbolic group with residually finite peripheral structure, and $B \curvearrowright I$ is a faithful action with infinite orbits. Integer index subfactors of $\L(G)$ are constructed from extensions of $G$. This result advances an open question of P.~de~la~Harpe~\cite{dlH95}.

\end{abstract}

\section{Introduction}

The theory of subfactors was initiated by Jones~\cite{Jon83}, who introduced the notion of the Jones index $[\mathcal{M}:\mathcal{N}]$ for an inclusion of $\rm{II}_1$ factors $\mathcal{N} \subseteq \mathcal{M}$. Defined as the Murray-von Neumann dimension of $L^2(\mathcal{M})$ as a left $\mathcal{N}$-module, the index measures the ``size" of $\mathcal{M}$ relative to $\mathcal{N}$. Jones's landmark result~\cite{Jon83} establishes that, for any fixed II$_1$ factor $\mathcal M$,  the collection $\mathscr{I}(\mathcal{M})$ of all finite Jones indices $[\mathcal{M}:\mathcal{N}]$ of subfactors $\mathcal{N} \subseteq \mathcal{M}$ is always contained in
\begin{equation}
    \left\{4\cos^2\left(\frac{\pi}{n}\right) \,|\, n \in \mathbb{N},\, n \ge 3\right\} \cup [4, \infty). \label{Jones indices}
\end{equation}
He also showed that all of these values are realized as Jones indices of subfactors of the hyperfinite $\mathrm{II}_1$ factor $\mathcal{R}$ ~\cite{Jon83}. In the same work, Jones asked which of these values can occur as Jones indices of irreducible subfactors of $\mathcal{R}$. This problem has proved to be extremely difficult and remains wide open.

\vskip 0.05in
For some $\rm{II}_1$ factors, their Jones index sets are identified. Using results from  \cite{IPP08}, Vaes constructed in \cite{Va09} a $\mathrm{II}_1$ factor $\mathcal{M}$ in which every finite-index subfactor $\mathcal{N} \subseteq \mathcal{M}$ is trivial; that is, every finite-index inclusion of factors $\mathcal{N} \subseteq \mathcal{M}$ is isomorphic to $\mathcal{N} \subseteq \mathcal{N} \otimes \mathbb{M}_n(\mathbb{C})$ for some $n \in \mathbb{N}$. As a result, the set $\mathscr{I}(\mathcal{M})$ of Jones indices of $\M$ is precisely $\mathbb{N}^2 = \{n^2 \,|\, n \in \mathbb{N}\}$, and the only finite index of an irreducible subfactor is $1$. 

\vskip 0.05in

The structure of the Jones index set remains largely mysterious for many natural classes of factors—one notable example being property (T) factors.
Answering a question from \cite{Jon83}, Popa proved in \cite{Pop86} via a general argument that for every property (T) II$_1$ factor $\mathcal{M}$, its Jones index set $\mathscr{I}(\mathcal{M})$ is countable. This, combined with the earlier results of Connes \cite{Co80} on the countability of symmetries of property (T) II$_1$ factors, provides further evidence that such factors are highly rigid objects, illustrating, for instance, a stark contrast with the case of amenable factors. Motivated by these initial results, P.\ de la Harpe proposed for study in \cite[Problem 5]{dlH95} the following problem: \emph{What are the possible values of the Jones index set $\mathscr I(\mathcal L(G))$ when $G$ is an ICC property (T) group?}

Although the problem has been around for three decades, significant progress was made only recently. In particular, it was shown in \cite{CIOS24, AMAKCK} that for broad classes of property~(T) wreath-like product groups and relatively hyperbolic groups $G$, the Jones index set $\mathscr{I}(\mathcal{L}(G))$ consists solely of integers. Motivated by these results, as well as recent advances on the Connes Rigidity Conjecture and its strong form by Popa \cite{CDHK24, CIOS23a, CIOS23b, CIOS24, AMAKCK}, it was conjectured in \cite{CIOS24} that $\mathscr{I}(\mathcal{L}(G)) \subseteq \mathbb{N}$ holds universally. Nevertheless, no \emph{explicit} computation of $\mathscr{I}(\mathcal{L}(G))$ is known for any ICC property~(T) group $G$. The primary goal of this paper is to present the first such computation.

\subsection{Main results}

To properly introduce our results we recall the concept of wreath-like product group introduced in \cite{CIOS23a}. Let $A$, $B$ be countable groups and let $B\curvearrowright I$ be an action on a countable set $I$. A group $G$ is a \emph{wreath-like product} of groups $A$ and $B\curvearrowright I$, written as $G\in\W\R(A,B\curvearrowright I)$, if $G$ is an extension of the form
\begin{equation}\label{ext1}
1\rightarrow \bigoplus_{i\in I}A_i \hookrightarrow  G \stackrel{\kappa}\twoheadrightarrow B\rightarrow 1,
\end{equation}
where $A_i\cong A$ and the action of $G$ on $A^{(I)}=\bigoplus_{i\in I}A_i$ by conjugation permutes the summands as  follows
$$gA_ig^{-1} = A_{\kappa(g)i}\;\;\; \text{ for all } g\in G,  i\in I.$$  

In \cite{CIOS23a}, a natural quotienting procedure was introduced in the context of group-theoretic Dehn filling, producing many examples of wreath-like product groups, including numerous groups with property (T). Some classes of these groups provided the first known examples of property (T) groups that are entirely reconstructible from their von Neumann algebras; in particular, they satisfy the Connes Rigidity Conjecture \cite{Con82}. Moreover, in \cite[Theorem 7.5]{CIOS23b}, for an even broader class of wreath-like product groups
$G$ and $H$ with abelian base, it was completely described all isomorphisms between $\L(G)$ and any amplification $\L(H)^t$ with $t>0$ solely
 in terms of isomorphisms between the underlying groups $G$ and $H$, together with multiplicative characters on $G$. As a result, all such groups give a positive answer to Jones’ Outer Automorphisms Conjecture from \cite[Problem 8]{Jon00}. 
 
 In this paper we establish the following nontrivial generalization of this result to all virtual $\ast$-isomorphisms.  Recall that two $\mathrm{II}_1$ factors $\mathcal{N}$ and $\mathcal{M}$ are called \emph{virtually isomorphic} if there exists $t > 0$ and an injective $\ast$-homomorphism 
$
\Theta : \mathcal{N} \to \mathcal{M}^t
$
with finite index image, i.e., $[\mathcal{M}^t : \Theta(\mathcal{N})] < \infty$. The map $\Theta$ is called a \emph{virtual $\ast$-isomorphism}.

\begin{thma}\label{vsymmetries1}
Let $A,C$ be non-trivial abelian groups. Let $B,D$ be nonparabolic ICC subgroups of finitely generated groups which are hyperbolic relative to a finite family of residually finite groups. 

\noindent Let $G\in\mathcal W\mathcal R(A,B \curvearrowright I)$ and $H\in\mathcal W\mathcal R(C,D\curvearrowright J)$ be property (T) groups, where $B \curvearrowright I$ and $D\curvearrowright J$ are faithful actions with infinite orbits.

\noindent Let $t>0$ be a scalar and let $\Theta:{\mathcal L}(G)\hookrightarrow {\mathcal L}(H)^t$ be any virtual $*$-isomorphism.

\noindent Then $t\in\mathbb N$, and 
we can find $m\in\mathbb N$ and for every $1\leq i\leq m$, a finite index subgroup $K_i\leqslant G$, an injective homomorphism $\gamma_i:K_i\rightarrow H$ with finite index image and a unitary representation $\rho_i:K_i\rightarrow\mathscr U_{s_i}(\mathbb C)$, for some $s_i\in\mathbb N$,  and a unitary $w\in\mathcal{L}(H)^t=\mathcal{L}(H)\otimes\mathbb M_t(\mathbb C)$  such that $\sum_{i=1}^m[G:K_i]s_i=t$ and
$$\text{$w\Theta(u_g)w^*=\emph{diag}(\emph{Ind}_{K_1}^G(\pi_{\gamma_1,\rho_1})(g),\ldots,\emph{Ind}_{K_m}^G(\pi_{\gamma_m,\rho_m})(g))$, for every $g\in G$.}$$

\noindent Here, $(u_g)_{g\in G}\subset \L(G)$ and $(v_h)_{h\in H}\subset \L(H)$ are the canonical unitaries, $\pi_{\gamma_i,\rho_i }(g)= v_{\gamma_i(g)}\otimes \rho_i(g)$ for all $g\in K_i$ and ${\rm Ind}^G_{K_i}(\pi_{\gamma_i, \rho_i}): G \ra \mathscr U(\L(K_i) \otimes  \mathbb M_{s_i[G:K_i]}(\mathbb C))$ is the canonical induced representation. 

\noindent Moreover, we have the following formula for the index of the image of $\Theta$

$$
[\mathcal{L}(H)^t:\Theta(\mathcal{L}(G))]=t\sum_{i=1}^m s_i[H:\gamma_i(K_i)].
$$

\end{thma}

This theorem should be compared with the recent embedding result of \cite[Theorem~5.6]{CIOS24}. Although our result is more restrictive—since it concerns only virtual isomorphisms, rather than the more general embedding framework considered there—it applies to a significantly broader class of wreath-like product groups. The advantage of working with this larger class is that, by results in \cite{CIOS23b}, it enables us to consider group factors with prescribed outer automorphism groups, which is a key ingredient in obtaining complete computations of the Jones index set for property~(T) group factors.

Regarding the proof of Theorem \ref{vsymmetries1}, we emphasize that while it follows the general strategy of \cite[Theorem~5.6]{CIOS24} at a technical level, it requires a new analysis based on finite-index techniques; this constitutes the main novelty of our approach.


By iteratively applying the basic construction \cite{Chr79,Jon83}, we observe that any subfactor $\mathcal{N} \subseteq \L(G)$ of finite Jones index $[\L(G) : \mathcal{N}] = t > 0$ gives rise to a $\ast$-embedding $\L(G) \hookrightarrow \L(G)^t$ whose image has finite index. In particular, the previous theorem implies that for every property (T) wreath-like product group $G$ as in the hypothesis, the Jones index set satisfies $\mathscr{I}(\L(G)) \subseteq \mathbb{N}$. Moreover, we show that for some of these groups, \emph{every} positive integer $t \in \mathbb{N}$ can be realized as the Jones index of an irreducible finite index subfactor of $\L(G)$. 
The key insight is that such subfactors can be constructed from finite index extensions of the group $G$. Extensions $H$ are chosen that $\L(G)\subset \L(H)$ are irreducible subfactors. The irreducibility is shown by expressing the von Neumann algebras $\L(H)$ as cocycle crossed products and applying \cite[Theorem 6.1]{Sut80B}. Applying the downward basic construction in \cite{Jon83, PP86} to $\L(G)\subset \L(H)$ yields irreducible subfactors of $\L(G)$ with index $[H:G]$.

\begin{thma}
     There is a continuum of ICC property (T) wreath-like product groups $(G_j)_{j\in J}$ as in the statement of Theorem \ref{vsymmetries1}, such that their II$_1$ factors $\mathcal L(G_j)$ are pairwise not stably isomorphic, and $\mathscr{I}(\mathcal L(G_j))=\mathbb{N}$. Moreover, all values in $\mathscr{I}(\mathcal L(G_j))$ can be realized as the Jones indices of irreducible subfactors of $\mathcal L(G_j)$.
\end{thma}
In addition to the above theorem, we construct examples of ICC property (T) groups whose Jones index sets are some proper subsets of $\mathbb{N}$ in Corollary \ref{proper set}. The construction uses the index $2$ and $3$ subfactor classification \cite{Jon83, GdlHJ89, Ocn88, Kaw95} and Goldman's type theorems \cite{Gol59, Jon83, Izu92}. However, we do not have a complete description of the Jones index sets.

It would be very interesting to understand which other intermediate subsets $\mathbb{N}^2 \subseteq S \subseteq \mathbb{N}$ can arise as Jones index sets of property (T) group factors. In particular, one may ask whether there exist ICC property (T) groups $G$ for which $\mathscr{I}(\mathcal{L}(G)) = \mathbb{N}^2$, and moreover, whether there are situations where such factors admit no non-canonical finite-index subfactors, as in \cite{Va09}.

\paragraph{Acknowledgments.} The first author was supported by NSF Grants DMS-2154637 and DMS-2452247, and the second author was supported by U.S. ARO Grant W911NF-23-1-0026. The authors are grateful to Dietmar Bisch for his extensive feedback, which significantly improved the exposition and overall quality of the paper. The authors also thank anonymous referees for their comments and suggestions, which greatly improved the quality of the writing and overall exposition. The second author  thanks Koichi Oyakawa, and Kai Toyosawa for helpful discussions.

\section{Preliminaries }

\subsection{Cocycle crossed product von Neumann algebras}Let $(\M, \tr)$ be a tracial von Neumann algebra endowed with a normalized trace $\tr$ and $\mathscr{U}(\M)$ and $\Aut(\M)$ be its unitary group and automorphism group. Each $w\in \sU(\M)$ induces a conjugation automorphism $\Ad(w)\in\Aut(\M)$ given by $\Ad(w)(x)=wxw^*$. Let $\Gamma$ be a discrete group. 
\begin{defn}\label{cocycle_def}
    A \textit{cocycle action} of $\Gamma$ on $\M$ is denoted by $\Gamma \ca^{ \alpha,\omega} \M$ and consist of maps $\alpha:\Gamma\rightarrow\Aut(\M)$  and $\omega:\Gamma\times \Gamma\rightarrow \sU(\M)$ satisfying the following relations:
    \begin{enumerate}
        \item $\alpha_1=\Id_\M$,
        \item $\alpha_g\alpha_h=\Ad(\omega_{g,h})\alpha_{gh}$ for all $g,h\in \Gamma$, and
        \item $\alpha_g(\omega_{h,k})\omega_{g,hk}=\omega_{g,h}\omega_{gh,k}$ for all $g,h,k\in \Gamma$.
    \end{enumerate}
    Note that $\alpha$ and $\omega$ are not necessarily group homomorphisms. The map $\omega$ is called a \textit{$2$-cocycle} for $\alpha$.
\end{defn}
The following is a well-known fact:
\begin{prop} If $\M$ is a $\rm{II}_1$ factor and $ \Gamma\ca^{\alpha, \omega} \M$ is a cocycle action as in Definition \ref{cocycle_def} then the following hold:
    \begin{enumerate}
        \item $\omega_{1,1}\in \mathbb{C}1$.
        \item $\omega_{1,1}=\omega_{1,k}=\omega_{g,1}$ for all $g,k\in \Gamma$.
    \end{enumerate}
\end{prop}
A $2$-cocycle $\omega$ is \textit{normalized} if $\omega_{1,g}=\omega_{g,1}=1$ for all $g\in \Gamma$. We can normalize $2$-cocycle by replacing $\omega_{g,h}$ by $\omega_{1,1}^*\omega_{g,h}$. Thus, throughout the paper we will assume that  all $2$-cocycles are normalized unless otherwise stated.
\begin{defn}
    Let $\Gamma \ca^{\alpha,\omega} \M$ be a cocycle action.  Then the  \textit{cocycle crossed product algebra} $\M\rtimes_{\alpha,\omega} \Gamma$ is the von Neumann algebra generated in  $\mathcal{B}(\ell^2(\Gamma,L^2(\M,\tr)))$ by the following operators:  
    \begin{enumerate}
        \item elements $x\in \M$ with the action given by $x(\xi)(h)=\alpha_h^{-1}(x)\xi(h)$ for all $\xi\in \ell^2(\Gamma,L^2(\M))$ and $h\in \Gamma$,
        \item unitaries $v_g$ given by $v_g(\xi)(h)= \alpha_h^{-1}(\omega_{g,g^{-1}h})\xi(g^{-1}h)$ for all $g\in G$.
    \end{enumerate}
\end{defn}
\noindent We notice these generators of $\M\rtimes_{\alpha,\omega}\Gamma$ satisfy the following relations:
    \begin{equation}
        v_gv_h=\omega_{g,h}v_{gh},\ \ \  v_gxv_g^*=\alpha_g(x),\ \ \ v_1=1. \label{crossed}
    \end{equation}

\subsection{Finite index inclusions of von Neumann algebras}

Given an inclusion of II$_1$ factors 
$\mathcal N \subseteq  \M$ its Jones index \cite{Jon83} is defined as the dimension $[\M : \mathcal N] = {\rm dim}_{\mathcal N} (L^2(\M))$. More generally, Pimsner and Popa discovered in \cite{PP86} a more probabilistic definition of finite index which then generalizes the index to an inclusion   $\mathcal N\subseteq \mathcal M$ of tracial von Neumann algebras. Specifically, consider $$c:=\inf \left \{\frac{\|\mathbb E_{\mathcal N}(x)\|^2_2}{\|x\|_2^2} \,|\, x\in \mathcal M_+, x\neq 0 \right \}.$$ 
Then we define the index of the inclusion $\mathcal N\subseteq \mathcal M$ as $[\mathcal M:\mathcal N]= c^{-1}$ with the convention that $\frac{1}{0}=\infty.$

We also say the inclusion $\mathcal N\subseteq \mathcal M$ has a finite (right) Pimsner-Popa basis $x_1,\ldots, x_n\in \mathcal M$ if these elements are $\mathcal N$-orthogonal and $\mathcal M = \sum_i \mathcal N x_i$. This implies that for every $x\in \mathcal M$ we have $x=\sum_{i}\mathbb E_{\mathcal N}(xx_i^*)x_i$. In \cite{PP86} it was proved that for II$_1$ factors $\mathcal N\subseteq \mathcal M$ the existence of a finite Pimsner-Popa basis for the inclusion  is equivalent to $[\mathcal M:\mathcal N]<\infty$. 

We continue by recording several basic facts from the literature concerning finite index inclusions of tracial von Neumann algebras which we will use  in the proofs of some of our main results. Recall that a von Neumann algebra $\M$ is called \textit{completely atomic} if $1$ is an orthogonal sum of minimal projections in $\M$. 

\begin{prop}\label{findex} Let $\mathcal N\subseteq \M$ be an inclusion of tracial von Neumann algebras with $[\M:\mathcal N]<\infty$. Then the following hold:

\begin{enumerate}

\item \cite{Jon83} If $p\in \mathcal N$ is a non-zero projection, then $[p\M p:p\mathcal N p]<\infty$. 
\item \cite[Relation 1.1.2(ii)]{Pop95} If $\mathcal N$ is a factor and $r\in \mathcal N'\cap \mathcal M$ is a non-zero projection, then $[r\mathcal Mr: \mathcal N r]<\infty$.


 \item \cite{Jon83} If $\mathcal N$ is a factor then ${\rm dim}_{\mathbb C} (\mathcal N' \cap \M) \leq  [\M : \mathcal N] + 1$.
\item  \cite[Relation 1.1.2(iv)]{Pop95} If  $\mathscr Z(\M)$ is completely atomic, then $\mathscr Z(\mathcal N)$ is completely atomic.


\item \cite{PP86} If $\M$ and $\mathcal N$ are factors, then there is a finite Pimsner-Popa basis for $\mathcal N\subseteq \M$.
\item \cite{Pop95} Let $\mathcal N \subseteq \Q \subseteq \mathcal M$ be any intermediate von Neumann subalgebra. If $\mathcal N \subseteq \mathcal M $ admits a finite Pimsner-Popa basis then so do the inclusions $\mathcal N\subseteq \Q$ and $\mathcal Q \subseteq \M$.
\item \cite[Lemma 3.1]{Po02} Assume that $\mathcal N\subseteq \M$ admits a finite Pimsner-Popa basis. Then $[\mathcal N'\cap \M : \mathscr Z(\M)]<\infty$. In particular, if $\M$ is a factor, then $\mathcal N'\cap \M$ is finite dimensional.

\item \cite{Jon83,PP86} If $\mathcal N \subseteq \R\subseteq \mathcal M$ is a von Neumann subalgebra,  then $\max\{ [\mathcal M :\mathcal R], [\mathcal R:\mathcal N]\}\leq [\M:\mathcal N] \leq [\M:\mathcal R][\mathcal R:\mathcal N]$. In particular,  $[\M:\mathcal R]$, $[\mathcal R:\mathcal N]<\infty$.
\item \cite[Lemmas 2.2.1-2]{Jon83} If $\mathcal N$ and $\mathcal M$ are II$_1$ factors then for every projection $p \in \mathcal N'\cap \M$ we have that  $$[p\M p : \mathcal N p]=[\M:\mathcal N] \tr_\M (p)\tr_{\mathcal N'}(p).$$
Moreover, for any partition of unity $(p_i)\subset  \mathcal N'\cap \M$ by projections we have that 
$$[\M:\mathcal N]= \sum_i \frac{1}{\tr_{\M}(p_i)}[p_i \M p_i : \mathcal Np_i].$$
\item If $\mathcal N:=\mathcal P \rtimes H\subseteq \mathcal P \rtimes G=: \M$ for an inclusion of discrete groups $H\leqslant G$ and trace preserving action $ G \ca \mathcal P$ on  tracial von Neumann algebra $\mathcal P$ then $[\mathcal M:\mathcal N]=[G:H]$.

\end{enumerate}
\end{prop}

\subsection{Construction of finite index subfactors from outer automorphisms}
In this subsection, we give a way to construct a finite index irreducible subfactor of an ICC group factor $\L(G)$ from a finite subgroup $\Gamma_0$ of the outer automorphism group $\Out(G)=\Aut(G)/\Inn(G)$. 

\begin{thm}[{\cite[Theorem IV.9.1]{Mac63}}]\label{extension}
    Let $G$ be a group with a trivial center and $\overline{\varphi}:\Gamma\rightarrow \Out(G)$ be a homomorphism. Then there is an extension $1\rightarrow G\rightarrow H\rightarrow \Gamma\rightarrow 1$ such that the induced homomorphism also denoted by $\overline{\varphi}:H/G\rightarrow \Out(G)$ satisfies
    $\overline{\varphi}(hG)=\Ad(h)\Inn(G)$.
\end{thm}
The description of $H$ is given in \cite[Section IV.8]{Mac63}. We briefly summarize the construction. For each $k\in \Gamma$, we choose $\varphi_k\in \Aut(G)$ whose image in $\Out(G)$ is $\overline{\varphi}_k$. Then for each pair of elements $k_1,k_2\in \Gamma$, we have  $\varphi_{k_1}\varphi_{k_2}=\Ad(g_{k_1,k_2})\varphi_{k_1k_2}$ for some $g_{k_1,k_2}$. Then $H$ is given by $$
H=\langle G, \{h_k\}_{k\in \Gamma}\ |\  h_kgh_k^{-1}=\varphi_k(g)\text{ for all  $g\in G$ and $k\in \Gamma$},\  h_{k_1}h_{k_2}=g_{k_1,k_2}h_{k_1k_2}\text{ for all $k_1,k_2\in \Gamma$}\rangle.
$$
\begin{rmk}\label{group preimage}
    \begin{enumerate}
        \item When $G$ has trivial center and $\overline{\varphi}$ is injective, then $H$ is isomorphic to the preimage of $\overline{\varphi}(\Gamma)\le \Out(G)$ in $\Aut(G)$ \cite[Section IV.9]{Mac63}.
        \item The map $\Gamma\times\Gamma\rightarrow G$ given by $(k_1,k_2)\mapsto g_{k_1,k_2}$ is a $G$-valued $2$-cocycle \cite[Section IV.8]{Mac63}, \cite[Section 3.1]{Sut80B}. Thus, $\L(H)$ is a cocycle crossed product $\L(G)\rtimes \Gamma$. In particular, the $3$-cocycle obstruction for the crossed product is trivial. 
    \end{enumerate}
\end{rmk}

\begin{thm}[{\cite[Theorem 6.1]{Sut80B}}]\label{Sutherland}
    Any cocycle crossed product of a $\rm{II}_1$ von Neumann algebra with a finite group is an ordinary crossed product.
\end{thm}

\begin{cor}\label{index}
    Let $\M$ be a II$_1$ factor with a normalized trace $\tr$, $\Gamma_0$ be a finite group and $\Gamma_0\ca^{ \alpha,\omega}\M$ be a cocycle action. 
    If $\alpha$ induces an injective homomorphism $\overline{\alpha}:\Gamma_0\rightarrow\Out(\M)$ from $\Gamma_0$ to the outer automorphism group $\Out(\M)=\Aut(\M)/\Inn(\M)$ of $\M$, then the following hold:
    \begin{enumerate}
        \item \cite{Sut80A, Sut80B} $\M\rtimes_{\alpha,\omega}\Gamma_0$ is a II$_1$ factor, its trace being given by 
        $$
        \tr_{\M\rtimes_{\alpha,\omega}\Gamma_0}\left(\sum_{k\in \Gamma_0}x_kv_k\right)=\tr_M(x_1).
        $$
        \item \cite{OK90, KY92} $\M\subset \M\rtimes_{\alpha,\omega}\Gamma_0$ is an irreducible inclusion.
        \item \cite{OK90, KY92} $[\M\rtimes_{\alpha,\omega}\Gamma_0:\M]=|\Gamma_0|$.
    \end{enumerate}
\end{cor}
Corollary \ref{index} follows directly from Theorem \ref{Sutherland} and well-known properties about the ordinary crossed product. It can also be verified by direct computation.

We end this section with a canonical, yet very useful, procedure for constructing finite index subfactors using outer automorphisms of a factor.  

\begin{prop}\label{cocycle construction}
    Let $G$ be an ICC group. Let $\Psi:\Aut(G)\rightarrow \Aut(\mathcal{L}(G))$ be the homomorphism defined by
    $$
    \Psi(\varphi)(u_g)=u_{\varphi(g)} \text{ for all } \varphi\in {\rm Aut}(G)\text{ and } g\in G.
    $$
    Furthermore, let $\overline{\Psi}:\Out(G)\rightarrow \Out(\mathcal{L}(G))$ be the canonical homomorphism induced by $\Psi$,
    $$\overline{\Psi}(\varphi {\rm Inn} (G))=\Psi(\varphi){\rm Inn}(\L (G))  \text{ for all } \varphi\in {\rm Aut}(G). $$
    
    \noindent Assume  $\Gamma_0\le \Out(G)$ is a finite subgroup and $\overline{\Psi}|_{\Gamma_0}$ is injective. Then there is an extension $1\rightarrow G\rightarrow H\rightarrow \Gamma_0\rightarrow 1$ such that $\L(G)\subseteq \L(H)$ is an irreducible subfactor with index $|\Gamma_0|$. In particular, $H$ is ICC. Moreover, there exist irreducible subfactors $\mathcal N \subseteq \L(G)$ such that $[\L(G): \mathcal N]=|\Gamma_0|$.
\end{prop}
\begin{proof}
    Since $G$ is ICC, it has a trivial center. Consider the extension $1\rightarrow G\rightarrow H\rightarrow \Gamma_0\rightarrow 1$ as in Theorem \ref{extension}. Then by Remark \ref{group preimage}, $\L(H)\cong \L(G)\rtimes_{\alpha, \omega} \Gamma_0$ for some cocycle action $\Gamma_0\ca^{\alpha,\omega} \M$. The injectivity of $\overline{\Psi}|_{\Gamma_0}$ and Corollary \ref{index} implies that $\L(G)\subseteq\L(H)$ is an irreducible subfactor with index $|\Gamma_0|$. Performing the downward basic construction $\mathcal N \subseteq \L(G)\subseteq \L(H)$ gives a subfactor of index $|\Gamma_0|$.
\end{proof}

\subsection{Popa's intertwining techniques}

In \cite{Pop06} introduced the following powerful criterion for the existence of intertwining between von Neumann algebras.

\begin{thm}\emph{\cite{Pop06}} \label{corner} Let $( \mathcal{M},\tr)$ be a  tracial von Neumann algebra and let $ \mathcal{P},  \mathcal{Q}\subseteq  \mathcal{M}$ be (not necessarily unital) von Neumann subalgebras. 
	Then the following are equivalent:
	\begin{enumerate}
		\item There exist projections $ p\in    \mathcal{P}, q\in    \mathcal{Q}$, a $\ast$-homomorphism $\Theta:p  \mathcal{P} p\rightarrow q \mathcal{Q} q$  and a partial isometry $0\neq v\in  \mathcal{M} $ such that $v^*v\leq p$, $vv^*\leq q$ and $\Theta(x)v=vx$, for all $x\in p  \mathcal{P} p$.
		\item For any group $ \mathcal G \subset \mathscr U( \mathcal{P})$ such that $ \mathcal G''=  \mathcal{P}$ there is no net $(u_n)_n\subset \mathcal G$ satisfying $\|\mathbb{E}_{  \mathcal{Q}}(xu_ny)\|_2\rightarrow 0$, for all $x,y\in   \mathcal{M}$.
		\item There exist finitely many $x_i, y_i \in  \mathcal{M}$ and $c>0$ such that  $\sum_i \| \mathbb{E}_{\mathcal{Q}}(x_i u y_i) \|^2_2 \geq c$, for all $u\in \mathscr U (\mathcal{P})$.
	\end{enumerate}
\end{thm} 
\vskip 0.02in
\noindent If one of the equivalent conditions from Theorem \ref{corner} holds, one says \emph{a corner of $ \mathcal{P}$ embeds into $ \mathcal{Q}$ inside $ \mathcal{M}$}, and writes $ \mathcal{P}\prec_{ \mathcal{M}} \mathcal{Q}$. If we moreover have that $ \mathcal{P} p'\prec_{ \mathcal{M}} \mathcal{Q}$, for any projection  $0\neq p'\in  \mathcal{P}'\cap 1_{ \mathcal{P}}  \mathcal{M} 1_{ \mathcal{P}}$, then one writes $ \mathcal{P}\prec_{ \mathcal{M}}^{\rm s} \mathcal{Q}$.
\vskip 0.05in

\noindent In the remaining part of the section we record a few technical intertwining results that will be used in the proofs of our main results. We start with the following well known result intertwining result in cocycle crossed product von Neumann algebras.

\begin{prop}\label{infindexnonint} Let $\Gamma \curvearrowright^{\alpha, \omega} \mathcal N$ be a cocycle action on a tracial von Neumann algebra and let $\mathcal M =\mathcal N\rtimes_{\alpha, \omega} \Gamma$ the corresponding cocycle crossed product von Neumann algebra. Let $T<\Gamma$ be an  infinite index subgroup and let $0\neq p\in \mathcal M$ be a projection. Then for any finite index von Neumann subalgebra $\mathcal P \subseteq p\mathcal M p$ we have $\mathcal P \nprec_{\mathcal M} \mathcal N\rtimes_{\alpha, \omega} T$. 
    
\end{prop}

For further use, we will also record the following result from \cite[Proposition 2.3]{CD18}

\begin{prop}[\cite{CD18}]\label{intfinindex}Let $\mathcal N \subseteq  \M$ be II$_1$ factors such that $\mathcal N' \cap \M = \mathbb C1$. Then $\M \prec_\M \mathcal N$ if and only if $[\M : \mathcal N ] < \infty$.
    
\end{prop}

Next, we will continue with several intertwining results that will be used in an essential way in deriving the main results in the subsequent sections. Some of them may be well-known to the experts, but for readers' convenience we decided to include complete proofs.

\begin{prop}\label{abeliansint} Let $\mathcal A \subseteq \mathcal B$ be abelian von Neumann algebras.  Then $\mathcal B\prec^{\rm s}_{\mathcal B} \mathcal A$ if and only if there exists a countable set of mutually orthogonal projections $(r_n)\subset \mathcal B$ such that $\mathcal B = \oplus_n \mathcal Ar_n$.
    
\end{prop}

\begin{proof} We will prove only the forward implication as the converse is straightforward. In this direction we first show that for every projection $0\neq q'\in \mathcal B$ there is a projection $0\neq r\in \mathcal B q'$ such that $\mathcal B r =\mathcal A r$. 

To see this, notice from assumptions, we can find nonzero projections $q\in \mathcal B q'$, $p \in \mathcal A$, a nonzero partial isometry $v\in \mathcal B$, and a unital $\ast$-isomorphism on the image $\Theta: \mathcal B q\ra \Theta(\mathcal B q)\subseteq \mathcal A p $ such that $vv^*\le p$, $v^*v\le q$, and $\Theta(x)v=vx$ for all $x\in \mathcal B q$. Since $\mathcal B$ is abelian, we further have $\Theta(x)vv^*= vxv^*= xvv^*$, for all $x\in \mathcal B q$. In particular, $\mathcal B vv^*= \Theta(\mathcal B q) vv^*\subseteq \mathcal A vv^*$. Since $\mathcal A \subseteq \mathcal B$ we further get $\mathcal B vv^*= \mathcal A vv^*$ and letting $r=vv^*\in \mathcal B$ we obtain the desired claim.
Using Zorn's Lemma there exists $(r_n)\subset \mathcal B$ a (countable) maximal (under set inclusion) family or mutually orthogonal projections such that $\mathcal B r_n= \mathcal A r_n$ for all $n$. From the claim this family is nonempty. 

Now, set $q'=1-\sum_n r_n$. If $q'\neq 0$ then one can find $0\neq r\leq q'$ such that $\mathcal B r=\mathcal A r$. Adding $r$ to the set $(r_n)$ would contradict the maximality of the latter. Hence $q'=0$ which completes our proof.  \end{proof}

\begin{prop}\label{essfinindex}  Let $\mathcal A \subseteq \mathcal N \subseteq \mathcal M$ be von Neumann algebras such that $\mathcal A\subseteq \mathcal N$ is a MASA, $\mathcal N \subseteq \mathscr N_{\mathcal M}(\mathcal A)''$, and the inclusion $\mathcal N \subseteq \mathcal M$ admits a finite Pimsner-Popa basis. Then the following hold: \begin{enumerate}
\item $\mathcal A'\cap  \mathcal M\prec^{\rm s}_{\mathcal A'\cap\mathcal M} \mathcal A $; in particular, $\mathcal A'\cap\mathcal M$ is type I;
\item If in addition $\mathcal N$ and $\mathcal M$ are II$_1$ factors then the inclusion $\mathscr Z(\mathcal A'\cap \mathcal M)\subseteq \mathcal A'\cap\mathcal M$ admits a finite Pimsner-Popa basis; and 
\item There is countable set of mutually orthogonal projections $(r_n)\subset \mathscr Z(\mathcal A'\cap\mathcal M)$ such that $\mathscr Z(\mathcal A'\cap\mathcal M) = \oplus_n \mathcal Ar_n$.
    \end{enumerate}
\end{prop}

\begin{proof} 1. Denote by $\mathcal G= \mathscr N_{\mathcal M}(\A)$. First we claim that  $(\mathcal G\cap \mathcal N)''=\mathcal N$. To see this, fix $g\in \mathcal G$ and let $\sigma_g: \A \ra \A$ be the automorphism given by $\sigma_g ={\rm ad}(g)$. Then for all $a\in \A$ we have $\sigma_g(a) g=g a$ and  applying the expectation $\mathbb E_{\mathcal N}$ we further get \begin{equation}\label{intertwining1}\sigma_g(a)\mathbb E_{\mathcal N} (g)= \mathbb E_{\mathcal N}(g) a.\end{equation} Letting $v= \mathbb E_{\mathcal N}(g)$, the intertwining relation \eqref{intertwining1} implies that $vv^*, v^*v\in \A '\cap \mathcal N =\mathcal A$. Consider the polar decomposition  $v=w|v|$ where $w\in \mathcal N$ is a partial isometry with $w^*w=s(|v|)$. Then \eqref{intertwining1} further implies that $\sigma_g(a)w=wa$ for all $a\in \A$. In particular, we have $w\A =\A w$ and $ww^*,w^*w\in \A$ and thus $w\in \mathscr{GN}_{\mathcal N}(\A)$---the groupoid normaliser of $\A$ in $\M$. This and $|v|= (v^*v)^{1/2}\in \A$ further imply that $v\in \mathscr {GN}_{\mathcal N}(\A)''$. Since this holds for every $g\in \mathcal G$ and $\mathbb E_{\mathcal N}$ is normal we get $\mathcal N = \mathbb E_{\mathcal N }(\mathcal G'')\subseteq \mathscr {GN}_{\mathcal N} (\A)'' \subseteq \mathcal N$ and hence $\mathscr {GN}_{\mathcal N} (\A)'' = \mathcal N$ However, from \cite[Lemma 2.2]{FSW10} we have $\mathscr {GN}_\mathcal N(\A)''\subseteq\mathscr N_{\mathcal N}(\A)''\subseteq \mathcal N$ and hence $\mathscr N_{\mathcal N}(\A)''=\mathcal N$, yielding our claim.

Let $\mathcal K = \mathscr U(\A'\cap\M)\mathscr N_{\mathcal N}(\mathcal A) $.  Since $\mathscr N_{\mathcal N} (\A)$ normalizes $\A$, it also normalizes $\A'\cap\M$. This further implies that $\mathcal K$ forms a subgroup of $\mathscr U(\M)$. Moreover, we have  $\mathcal N\subseteq \mathcal K''=:\Q\subseteq \M$.  
Since $\mathcal N\subseteq \mathcal M$ has a finite Pimsner-Popa basis then by Proposition \ref{findex} (6) so does $\mathcal N\subseteq \mathcal Q$.  Let $x_1,\ldots ,x_s\in \mathcal Q$ be a Pimsner-Popa basis over $\mathcal N$. For every $x\in \mathcal Q$ we have $x=\sum_i \mathbb E_{\mathcal N }(xx_i^*)x_i$. Using $\mathcal N$-orthogonality, we further get  
\begin{equation}\label{part1}\|x\|_2^2=\sum_{i=1}^s \|\mathbb{E}_{\Nn}(xx_i^*)\|_2^2.\end{equation}
Fix $z\in \mathscr Z (\A'\cap\M )$ an arbitrary projection. Fix $1>\varepsilon>0$. 
 By Kaplanski Density Theorem, there is $y_i=\sum_{j\in J_i}c_{j}u_{j}g_{j}\in \spa  \mathcal K$ such that $\|x^*_i-y_i\|_2<\frac{\varepsilon\|z\|^2_2} {3s\max\{\|x_i\|_\infty,1\}}$ and $\|y_i\|_\infty\leq \|x^*_i\|_\infty$, where $c_j\in \mathbb C$,   $u_j \in \mathscr U(\mathcal A'\cap \mathcal M)$, $g_j\in \mathscr N_{\mathcal N}(\A)$ and $J_i$ is a finite set.  Then for all $x\in (\Q)_1$ we have
\begin{equation}\begin{split}\label{part2}
&\|\mathbb{E}_{\Nn}(xx_i^*)\|^2_2\le \left(\|\mathbb{E}_{\Nn}(x(x_i^*-y_i))\|_2 + \|\mathbb{E}_{\Nn}(xy_i)\|_2\right)^2\leq \left(\frac{\varepsilon \|z\|^2_2}{3s \max\{\|x_i\|_\infty,1\}} + \|\mathbb E_{\Nn}(xy_i)\|_2\right )^2\\
& =\frac{\varepsilon^2\|z\|^4_2} {9s^2\max\{\|x_i\|_\infty,1\}^2} +  \frac{2\varepsilon\|z\|^2_2 \|\mathbb E_{\Nn}(xy_i)\|_2} {3s\max\{\|x_i\|_\infty,1\}}+ \|\mathbb E_{\Nn}(xy_i)\|^2_2\leq \frac{\varepsilon \|z\|_2^2}{s} +\|\mathbb E_{\Nn}(xy_i)\|^2_2\end{split}\end{equation}
Using Cauchy-Schwarz inequality and letting $c =\max_i (\sum_{j\in J_i} |c_j|^2)^{1/2}$ we have  
\begin{equation}\label{part3}
\|\mathbb E_{\mathcal N }(x y_i)\|^2_2\le \left(\sum_{j\in J_i}|c_{j}|\|\mathbb E_{\mathcal N }(x u_{j}g_{j})\|_2\right)^2\le c\sum_{j\in J_i}\|\mathbb E_{\mathcal N }(x u_{j}g_{j})\|^2_2=c\sum_{j\in J_i} \|\mathbb E_{\mathcal N}(xu_j)\|^2_2.
\end{equation}
Altogether, \eqref{part1},\eqref{part2} and \eqref{part3} imply that for all $x\in (\A'\cap \M)z$ we have

\begin{equation}\label{final ineq}\begin{split}
 \|x\|^2_2&\leq \sum^s_{i=1}\left(\frac{\varepsilon \|z\|^2_2}{s}+ c\sum_{j\in J_i} \|\mathbb E_{\mathcal N }(x u_j ) \|^2_2\right)  \leq \varepsilon \|z\|^2_2 +c\sum^s_{i=1}\sum_{j\in J_i} \|\mathbb E_{\mathcal N }(x u_j )\|^2_2\\
 &= \varepsilon \|z\|_2^2+ c\sum^s_{i=1}\sum_{j\in J_i} \|\mathbb E_{\mathcal N }(x u_j )\|^2_2= \varepsilon\|z\|^2_2+ c\sum^s_{i=1}\sum_{j\in J_i}\|\mathbb E_\A(xu_j)\|^2_2.\end{split}
\end{equation}
The last equality in \eqref{final ineq} holds because of the commuting square condition from \cite[Theorem 4.3.1]{GdlHJ89}, i.e. $\mathbb E_{\mathcal N}\circ\mathbb E_{\A'\cap \M}=\mathbb E_\A$. Letting $x\in \mathscr U((\A'\cap\M)z)$ in inequality \eqref{final ineq} and $\varepsilon>0$ sufficiently small,  part 3.\ in Theorem \ref{corner} further implies that  $(\mathcal A'\cap  \mathcal M)z\prec_{\mathcal A'\cap\mathcal M} \mathcal A $. As $z\in \mathscr Z(\A'\cap\M)$ was arbitrary, we conclude that $\A'\cap \M\prec^{\rm s}_{\A'\cap\M} \A$.

For the remaining part, notice this intertwining further implies that for every projection $0\neq p\in \mathcal A'\cap \mathcal M$ there exists an abelian  projection $0\neq r\in \mathcal A$ such that $r\leq z(p)$, where $z(p)$ is the central carrier of $p$ inside $\mathcal A'\cap\mathcal M$. Therefore $z(r)\leq z(p)$ and hence $0\neq z(r)=z(r)z(p)$. By Comparison Theorem this further yields that $r$ and $p$ admit nonzero von Neumann equivalent subprojections. Since being abelian is preserved under both von Neumann equivalence and taking subprojections we conclude that $p$ has a nonzero, abelian subprojection. As $p$ was arbitrary it follows that $\mathcal A'\cap\mathcal M$ is type I. 

\noindent 2. First we argue that $\mathscr Z(\A'\cap\M)\subseteq \A'\cap \M$ has finite index. By \cite[Theorem 4.3.1]{GdlHJ89} again we have that $\mathbb E_{\mathcal N}\circ \mathbb E_{\A'\cap \M}= \mathbb E_{\A'\cap \M}\circ \mathbb E_{\mathcal N}=\mathbb E_\A$. Since $\mathcal N\subseteq \M$ has finite index there exists $c>0$ such that $\|\mathbb E_{\mathcal N}(x)\|_2^2\geq c\|x\|^2_2$, for all $x\in (\M)_+$. Using this together with $\mathbb E_{\mathcal N }\circ \mathbb E_{\A'\cap\M}= \mathbb E_\A$ and $\A \subseteq \mathscr Z(\A'\cap\M)$, for every $x\in (\A'\cap \M)_+$ we have   \begin{equation*}
c\|x\|_2^2\leq\|\mathbb E_{\mathcal N}(x)\|^2_2 =\|\mathbb E_{\A}(x)\|_2^2\leq \|\mathbb E_{\mathscr Z(\A'\cap\M)}(x)\|_2^2 .
\end{equation*} This proves our claim. Finally, as $\mathcal A'\cap\M$ is type I, the conclusion follows directly from Proposition \ref{typeI}.

\noindent 3. Part 1.\ clearly implies that $\mathscr Z(\mathcal A'\cap\mathcal M)\prec^{\rm s}_{\mathcal A'\cap\mathcal M} \mathcal A$. Combining this with \cite[Proposition 3.2]{CFQT24} we further have that $\mathscr Z(\mathcal A'\cap\mathcal M)\prec^{\rm s}_{\mathscr Z(\mathcal A'\cap\mathcal M)} \mathcal A$. Thus, the conclusion follows from Proposition \ref{abeliansint}.\end{proof}

\begin{prop}\label{typeI} Let $(\mathcal M, \tr)$  be  a tracial von Neumann algebra such that $\mathscr Z(\mathcal M)\subseteq \mathcal M$ has finite index. Then $\mathcal M$ is type I and there exists a finite Pimsner-Popa basis for the inclusion 
$\mathscr Z(\mathcal M)\subseteq \mathcal M$.

\end{prop}
\begin{proof}
From the assumption there exists a constant $c>0$ such that $\|\mathbb E_{\mathscr Z(\mathcal M)}(x)\|^2_2\geq c\|x\|^2_2$ for all $x\in \mathcal M_+$. 
This further implies that for any projection $p\in \mathcal M$ we have \begin{equation}\label{finindexineq1}\tr (\mathbb E_{\mathscr Z(\mathcal M)}(p) \mathbb E_{\mathscr Z(\mathcal M)}(p) ) \geq c\tr(p)= c \tr(\mathbb E_{\mathscr Z(\mathcal M)}(p) ).\end{equation}

\noindent Notice that $\mathcal M= \mathcal M r \oplus \mathcal M q$, for $r,q\in \mathscr Z(\mathcal M)$ with $r+q=1$ where $\mathcal M r$ is type I and $\mathcal Mq$ is type II. Since $\mathcal M q$ is type II, for every $n\in\mathbb N$ one can find mutually orthogonal, von Neumann equivalent projections $q_i\in\mathcal M q$ for $1\leq i\leq n$ such that $q=q_1+\cdots +q_n$. Fixing $1\leq i\leq n$ we see that $n \mathbb E_{\mathscr Z(\mathcal M)}(q_i) = \mathbb E_{\mathscr Z(\mathcal M)}(q)$. 
Using this together with  inequality \eqref{finindexineq1} for $p=q_i$ we see that
\begin{equation*}
    0\leq  \tr(q)=\tr(\mathbb E_{\mathscr Z(\mathcal M)}(q) )=n \tr(\mathbb E_{\mathscr Z(\mathcal M)}(q_i))\leq \frac{n}{c} \tr (\mathbb E_{\mathscr Z(\mathcal M)}(q_i) \mathbb E_{\mathscr Z(\mathcal M)}(q_i) ) = \frac{1}{cn}\tr (\mathbb E_{\mathscr Z(\mathcal M)}(q) \mathbb E_{\mathscr Z(\mathcal M)}(q) ).
\end{equation*}
Letting $n\nearrow \infty$ above and using the faithfulness of $\tr$ we get $q=0$. Hence  $\mathcal M$ is type I. 
\vskip 0.04in
Therefore, after an isomorphism we can assume that $\mathcal M= \oplus_i \mathcal A_i\otimes \mathbb M_{n_i}(\mathbb C)$ where  $\mathcal A_i$ is abelian and $(n_i)$ is a sequence of of distinct positive integers. Next we show $(n_i)$ is finite. Assume by contradiction, $(n_i)$ is infinite. Thus we can assume that $n_i\nearrow \infty$ as $i \nearrow \infty$. Thus for every $i$ there is a nonzero projection $z_i\otimes p_i\in \mathcal A_i \otimes \mathbb M_{n_i}(\mathbb C)$ where $\tr(p_i)=1/n_i$. Since $\mathscr Z(\mathcal M)=\oplus_i \mathcal  A_i$ we get that 

\begin{equation*}
    \frac{\|z_i\|_2^2}{n^2_i}=\|z_i \otimes \mathbb E_{\mathcal A_i\otimes 1}(1\otimes p_i)\|^2_2=\|\mathbb E_{\mathscr Z(\mathcal M)}(z_i\otimes p_i)\|^2_2\geq c\|z_i\otimes p_i\|_2^2=c\|z_i\|_2^2\| p_i\|_2^2=\frac{c\|z_i\|_2^2}{n_i}.\end{equation*}
However, this leads to a contradiction when $n_i$ is sufficiently large. 
Since $(n_i)$ is finite, using the structure of $\mathcal M$ we have that $\mathscr Z(\mathcal M)\subseteq \mathcal M$ has a finite Pimsner-Popa basis. 
\end{proof}

\begin{prop}\label{finiteindexnorm} Let $\mathcal A\subseteq \mathcal N\subseteq \mathcal M$ be tracial von Neumann algebras where $\mathcal N\subseteq \mathcal M$ is a finite index inclusion of II$_1$ factors  and $\mathcal A \subseteq \mathcal N$ is a Cartan subalgebra. Consider a von Neumann subalgebra $\mathcal A \subseteq \mathcal B \subseteq \mathcal A'\cap \mathcal M$ such that $\mathcal B\subseteq \mathcal M$ is a MASA. Then the following hold:
\begin{enumerate}
\item The inclusion $\mathcal A \subseteq \mathcal A'\cap \mathcal M$ has finite index and  $\mathcal A'\cap \mathcal M$ is a type I. Moreover for every $\varepsilon>0$ there is $z \in \mathscr Z(\mathcal A'\cap\mathcal M)$ with $\tr(z)\geq 1-\varepsilon$ such that the inclusion $\mathcal A z\subseteq( \mathcal A'\cap \mathcal M) z$ admits a finite Pimsner-Popa basis. 

\item The center is completely atomic, i.e. $\mathscr Z(\mathscr N_{\mathcal M}(\mathcal B)'') = \oplus_n \mathbb C z_n$. Also for every $n\in \mathbb N$ we have that $\mathscr N_{\mathcal M}(\mathcal B)''z_n \subseteq z_n \mathcal M z_n$ is an irreducible, finite index inclusion of II$_1$ factors.  
\end{enumerate}
    
\end{prop}

\begin{proof} 1. This follows directly from Proposition \ref{essfinindex}. 

\noindent 2. Next we prove the following 

 \begin{equation}\label{cont1}\mathscr N_{\mathcal M}(\mathcal A)\subseteq \mathscr N_{\mathcal M}( \mathcal A'\cap\mathcal M)\subseteq \mathscr U(\mathcal A'\cap\mathcal M) \mathscr N_{\mathcal M}(\mathcal B).\end{equation}
    
\noindent The first inclusion follows directly from definitions. To see the second containment,  fix $u\in \mathscr N_{\mathcal M}(\mathcal A'\cap\mathcal M)$. Therefore, both $u\mathcal B u^*$ and $\mathcal B$ are MASAs in $\mathcal A'\cap \mathcal M $. Since by \cite[Theorem 6.5.5]{Dix81} in a type I von Neumann algebra any two MASAs are unitarily conjugate, one can find $v\in  \mathscr U(\mathcal A'\cap\mathcal M)$ such that $u\mathcal B u^*=v\mathcal B v^*$. Hence $v^*u\mathcal B u^*v=\mathcal B$ and thus $w= v^*u \in \mathscr N_{\mathcal M}(\mathcal B)$. As $u= vw$, we get the desired claim. 

Denote by $\mathcal Q= \mathscr N_{\mathcal M}(\mathcal B)''$. Also, since $\B\subseteq \A'\cap\M$ and $\B'\cap \M=\B$ we see that $\mathscr Z(\A'\cap \M) =(\A'\cap\M)'\cap (\A'\cap\M)\subseteq \B' \cap \M =\B\subset \Q$.
As $\mathcal A'\cap \mathcal M$ is type I, for every $\varepsilon>0$ one can find a projection $q\in \mathscr Z(\mathcal A'\cap \mathcal M) $ such that $\tr(q)>1-\varepsilon$ and the inclusion $\mathcal A q\subseteq (\mathcal A'\cap\mathcal M) q$ admits a finite Pimsner-Popa basis $x_1, \ldots, x_n\in  (\mathcal A'\cap\mathcal M) q$; in other words, \begin{equation}\label{PPbasis1}(\mathcal A'\cap\mathcal M )q = \sum^n_{i=1} x_i \mathcal A q.\end{equation} Using \eqref{cont1} \eqref{PPbasis1}, and $\A \subseteq \B \subseteq \Q$  we can see that \begin{equation}\begin{split}q \mathscr N_\M(\A)&\subseteq q \mathscr U(\A'\cap\M) \mathscr N_\M(\B)= \mathscr U(\A'\cap\M )q \mathscr N_\M(\B)\subseteq \\& \sum^n_{i=1} x_i \A q \mathscr N_\M(B)\subseteq \sum^n_{i=1} x_i \A q \Q\subseteq \sum^n_{i=1} x_i q \Q.\end{split}\end{equation}  Hence $q({\rm span} \mathscr N_\M (\A) )\subseteq \sum^n_{i=1} x_i q\Q$ and taking $\|\cdot\|_2$-limits we get $L^2(q \mathscr N_\M(\A)'')\subseteq \overline{\sum^n_{i=1} x_i q \Q}^{\|\cdot\|_2}$. This further implies \begin{equation}\label{fingen2}
    L^2(q\mathcal N q)\subseteq L^2(q \mathscr N_{\mathcal M}(\mathcal A)'' q)\subseteq \overline{ \sum^n_{i=1} x_i  q \mathcal Q q}^{\|\cdot\|_2}.
\end{equation}
As explained before we have  $q\in \mathscr Z (\A'\cap\M)\subseteq \Q$. 
Since $\mathcal N\subseteq \mathcal M$ has finite index, by Proposition \ref{findex}, so is $q\mathcal N q\subseteq q\mathcal M q$ and hence the latter inclusion admits a Pimsner-Popa basis $y_1,\ldots , y_m\in q\mathcal M q$. This together with \eqref{fingen2} show that $q\mathcal M q \subseteq \overline{ \sum^{m,n}_{j=1,i=1} (y_j x_i)  q \mathcal Q q}^{\|\cdot\|_2}$ and hence 
 \begin{equation}\label{fingen1}
    L^2(q\mathcal M q)= \overline{ \sum^{m,n}_{j=1,i=1} (y_j x_i)  q \mathcal Q q}^{\|\cdot\|_2}  .
\end{equation} 
Thus, $L^2(q\mathcal M q)$ is a finitely generated right $q\mathcal Q q$-module. Using a Gram-Schmidt argument one can find finitely many $q\mathcal Q q$-orthogonal elements $\xi_1, \ldots, \xi_s \in L^2(q\mathcal M q)$   such that  for any $y\in q\mathcal M q$ we have $y= \sum^s_{i=1} \xi_i \langle \xi_i,y\rangle_{q\Q q}$. This further implies that $q\Q q$$\subseteq q\M q$ admits a finite Pimsner-Popa basis in $q\M q$. Since $\B \subseteq \M$ is a MASA one can check that $\Q'\cap\M\subseteq \B'\cap \M=\B\subseteq\Q$. Thus, we have $\Q'\cap \M =\mathscr Z(\Q)$. Since $\M$ is a II$_1$ factor then Proposition \ref{findex} (7) further implies that $\mathscr Z(\Q) q=(\Q'\cap\M) q\subseteq q\Q q'\cap q\M q$ is an abelian  finite dimensional algebra. Since this holds for every projection $q$ with ${\rm tr}(q)\geq 1-\varepsilon$ and every $\varepsilon>0$, one can find a countable family of orthogonal projections $(z_n)_n \subset \mathscr Z(\mathcal Q)$ satisfying $\mathscr Z(\mathcal Q) = \oplus_n \mathbb C z_n$, yielding the first assertion of the conclusion. This  together with the above and Proposition \ref{intfinindex} further yield that $\mathcal Q z_n \subseteq z_n\mathcal M z_n$ is an irreducible inclusion of II$_1$ factors.

Fix an arbitrary $n\in \mathbb N$. In the last part we argue that $\mathcal Qz_n \subseteq z_n \M z_n$ has finite index. In the previous paragraph we proved that for every $\varepsilon>0$ and every projection $q\in \mathscr Z(\A'\cap\M)$ with $1-\tau(q)\leq \varepsilon$ the inclusion $q\Q q \subseteq q\M q$ has finite Pimsner-Popa basis. Pick $q$ large enough such that $qz_n\neq 0$. Since $q\Q q\subset q\M q$  has finite Pimsner-Popa basis then so is $q\Q z_n q\subseteq qz_n \M qz_n$. This implies that $z_n\M z_n\prec_{z_n\M z_n} \Q z_n$ and since $\Q z_n \subseteq z_n \M z_n$ is irreducible, Proposition \ref{intfinindex} further yields that $\Q z_n \subseteq z_n\M z_n$ has finite index. \end{proof}

\section{Virtual isomorphisms between II$_1$ factors  associated with property (T) wreath-like product groups}

This section is mainly devoted to a result which describes the structure of all virtual $\ast$-isomorphisms between II$_1$ factors associated with property (T) groups. The groups involved will be the wreath-like product groups introduced in \cite{CIOS23a}. We start by recalling their definition.

\begin{defn}\label{wlp}
Let $A$, $B$ be arbitrary groups, $I$ an abstract set, $B\curvearrowright I$ a (left) action of $B$ on $I$. A group $G$ is a \emph{wreath-like product} of groups $A$ and $B$ corresponding to the action $B\curvearrowright I$ if $G$
is an extension of the form
\begin{equation}\label{ext2}
1\rightarrow \bigoplus_{i\in I}A_i \hookrightarrow  G \stackrel{\kappa}\twoheadrightarrow B\rightarrow 1,
\end{equation}
where $A_i\cong A$ and the action of $G$ on $A^{(I)}=\bigoplus_{i\in I}A_i$ by conjugation satisfies the rule
$$gA_ig^{-1} = A_{\kappa(g)i}\;\;\; \text{ for all } i\in I.$$  
We call $A$ the \textit{base} of a wreath-like product $G\in\W\R(A,B\curvearrowright I)$. We also call the quotient group homomorphism  $G \stackrel{\kappa}\twoheadrightarrow B$ the \emph{canonical epimorphism}.

If the action $B\curvearrowright I$ is regular (i.e., free and transitive), we say that $G$ is a \emph{regular wreath-like product} of $A$ and $B$. The set of all wreath-like  products of groups $A$ and $B$ corresponding to an action $B\curvearrowright I$ (respectively, all regular wreath-like products) is denoted by $\WR(A, B\curvearrowright I)$ (respectively, $\WR(A,B)$). 
\end{defn}

To introduce our statement we need one more definition. Two II$_1$ factors $\mathcal M$ and $\mathcal N$ are called \emph{virtually isomorphic} if there is $t>0$ and a $\ast$-embedding  $\Theta:\mathcal M\rightarrow \mathcal N^t$ whose image $\Theta(\mathcal M )\subseteq \mathcal N^t$ has finite index; such  $\Theta$ is called a \emph{virtual $\ast$-isomorphism}.

The proof of the next result closely follows that of \cite[Theorem 5.1]{CIOS24}, incorporating certain elements from \cite{CIOS23b}. We recommend that the reader consult the proof of \cite[Theorem 5.1]{CIOS24} and its corresponding preliminaries in advance. For completeness, however, we present all necessary details here, including repeating verbatim some of the arguments used in \cite[Theorem 5.1]{CIOS24}.

\begin{thm}\label{symmetries}
Let $A,C$ be non-trivial abelian groups. Let $B,D$ be nonparabolic ICC subgroups of finitely generated groups which are hyperbolic relative to a finite family of residually finite groups. 

\noindent Let $G\in\mathcal W\mathcal R(A,B \curvearrowright I)$ and $H\in\mathcal W\mathcal R(C,D\curvearrowright J)$ be property (T) groups, where $B \curvearrowright I$ and $D\curvearrowright J$ are faithful actions with infinite orbits.

\noindent Let $t>0$ be a scalar and let $\Theta:{\mathcal L}(G)\hookrightarrow {\mathcal L}(H)^t$ be any virtual $*$-isomorphism. 

\noindent Then $t\in\mathbb N$ and there are $t_1,\ldots, t_{m}\in\mathbb N$ with $t_1+\cdots+t_{m}=t$, for some ${m}\in\mathbb N$, a finite index subgroup $K<G$, an injective homomorphism $\gamma_i:K\rightarrow H$ with finite index image $\gamma_i(K)\leqslant H$, and a unitary representation $\rho_i:K\rightarrow\sU_{t_i}(\mathbb C)$, for every $1\leq i\leq {m}$, and a unitary $w\in \mathcal L(H)^t=\mathcal L(H)\overline{\otimes}\mathbb M_t(\mathbb C)$ such that 
$$\text{$w\Theta(u_g)w^*=\emph{diag}(v_{\gamma_1(g)}\otimes\rho_1(g),
\ldots, v_{\gamma_{m}(g)}\otimes\rho_{m}(g))
$, for every $g\in K$.}$$
\end{thm}

\begin{proof} Let $\kappa_G:G\rightarrow B$ and $\kappa_H:H\rightarrow D$ be the canonical epimorphisms. For $b\in B$ and $d\in D$, fix $\widehat{b}\in G$ and $\widehat{d}\in H$ such that $\kappa_G(\widehat{b})=b$ and $\kappa_H(\widehat{d})=d$.
Denote by $\M=\mathcal L(G)$, $\P=\mathcal L(A^{(I)})$, $\mathcal N=\mathcal L(H)$ and $\Q=\mathcal L(C^{(J)})$. Let ${n_t}$ be the smallest integer such that ${n_t}\geq t$. Denote $\sN=\Nn\overline{\otimes}\mathbb M_{n_t}(\mathbb C)$ and $\widetilde{\mathcal Q}=\Q\overline{\otimes}\mathbb D_{n_t}(\mathbb C)$.
Let $p\in\widetilde{\mathcal Q}$ be a projection such that $(\tr\otimes\text{Tr})(p)=t$ and identify $\Nn^t=p\sN p$. For $1\leq i\leq {n_t}$, let $e_i=\textbf{1}_{\{i\}}\in\mathbb D_{n_t}(\mathbb C)$.
Let $$\sR:=\Theta(\P)'\cap p\sN p.$$

Since $H$ has property (T), so is $D$.
By applying \cite[Lemma 3.31]{CIOS23b} to $D$, we find a short exact sequence $$1\ra  S \ra D \overset{\kappa_D}{\ra} L\ra 1,$$ where $ S$ is either trivial or a nontrivial free product, $S= S_1\ast  S_2$ with $|S_1|\geq 2$ and $|S_2|\geq 3$, and $L$ is a non-elementary subgroup of a hyperbolic group. Let $\kappa=\kappa_D\circ\kappa_H:H\rightarrow L$ and denote by $T=\ker(\kappa)$. Then $\kappa_H(T)=\ker(\kappa_D)= S$, $C^{(J)}<T<H$, and $T<H$ has infinite index.

Next, we establish the following:
\begin{claim}\label{theta(P)'}
$\Theta(\P )\prec^{\rm s}_{\mathscr N}{\mathcal L}(T)\otimes \mathbb M_{n_t}(\mathbb C)$.
\end{claim}

\noindent \emph{Proof of Claim \ref{theta(P)'}.} Fix a projection $0\neq r\in \mathscr N_{p\mathscr N p}(\Theta (\mathcal P))' \cap p\mathscr N p$. Notice $\Theta(\mathcal P)r$ is abelian and hence amenable. Moreover, its normalizer satisfies    $\Theta(\mathcal M)r \subseteq \mathscr N_{p\mathscr N p }(\Theta (\P))'' r\subseteq r\mathscr N r$, and since $\Theta(\mathcal M)r \subseteq r\mathscr Nr $ has a finite index, it follows that $\mathscr N_{p\mathscr N p }(\Theta (\P))'' r\subseteq r\mathscr N r$ has finite index as well. In particular, as $\mathscr N$ has property (T), so is $\mathscr N_{p\mathscr N p }(\Theta (\P))'' r$ \cite{Pop86,CI18}. Furthermore, since $T<H$ has infinite index, it follows from Proposition \ref{infindexnonint}  $\mathscr N_{p\mathscr N p }(\Theta (\P))'' r\nprec_ {\mathscr N} \mathcal L(T)\otimes \mathbb M_{n_t}(\mathbb C)$. Altogether, these in combination with \cite[Theorem 6.4 (a)]{CIOS23b} yield that $\Theta(\P ) r\prec_{\mathscr N}{\mathcal L}(T)\otimes \mathbb M_{n_t}(\mathbb C)$. Since this holds for all $0\neq r\in \mathscr N_{p\mathscr N p}(\Theta (\mathcal P))' \cap p\mathscr N p$, using \cite[Lemma 2.4 (b)]{DHI19} we get the desired claim. $\hfill\blacksquare$


\vskip 0.08in

We continue with the following: 

\begin{claim}\label{theta(P)''}
$\Theta(\P )\prec^{\rm s}_{\mathscr N} \mathcal Q \otimes \mathbb D_{n_t}(\mathbb C)$.
\end{claim}

\noindent \emph{Proof of Claim \ref{theta(P)''}.} Fix a projection $0\neq r\in \mathscr N_{p\mathscr N p}(\Theta (\mathcal P))' \cap p\mathscr N p =\mathscr Z(\mathscr N_{p\mathscr N p}(\Theta (\mathcal P))'')$ . Note  $\Theta(\P )\subset \Theta (\mathcal M)$ is a Cartan subalgebra and $\Theta(\mathcal M) \subseteq p\mathscr N p$ has finite index. Consider a von Neumann subalgebra $\Theta(\P )\subseteq \mathscr B \subseteq \mathscr R$ such that $\mathscr B \subseteq p\mathscr N p$ is a MASA. Moreover, by Proposition \ref{finiteindexnorm} we have $\mathscr Z(\mathscr N_{p\mathscr N p}(\mathscr B)'')=\oplus_{n}\mathbb C z_{n}$ and
the inclusion $\mathscr N_{p\mathscr N p}(\mathscr B)''z_{n} \subseteq z_{n}\mathscr N  z_{n}$ has finite index for all $ {n}\in\mathbb N$. Since $\mathscr Z(\mathscr N_{p\mathscr N p}(\Theta (\mathcal P))'') \subseteq \mathscr B$ and $\Theta(\P )\subseteq \mathscr B$ is finite index, by Proposition \ref{findex}, so is $\Theta(\P )r z_{n}\subseteq \mathscr B r z_{n}$. Thus, Claim \ref{theta(P)'} implies that  $\mathscr B rz_{n} \prec_{\mathscr N} {\mathcal L}(T)\otimes \mathbb M_{n_t}(\mathbb C)$.  \cite[Proposition 3.6]{CIK15} gives nonzero projections $e\in \mathscr B r z_{n},q\in {\mathcal L}(T)\otimes \mathbb M_{n_t}(\mathbb C)$, a MASA $\R\subset q({\mathcal L}(T) \otimes \mathbb M_{n_t}(\mathbb C)) q$, a projection $q'\in \R'\cap q\mathscr N  q$ and $u\in\sU(\mathscr N)$ such that the inclusion $\sN_{q ({\mathcal L}(T)\otimes \mathbb M_{n_t}(\mathbb C) )q}(\R)''\subset q ({\mathcal L}(T)\otimes \mathbb M_{n_t}(\mathbb C))  q$ has finite index and $u\mathscr B eu^*=\R q'$. Since $\kappa_H(T)= S$ is a nontrivial free product and $\ker(\kappa_H)=C^{(J)}$,  \cite[Theorem 6.1 (c)]{CIOS23b} gives that $\R\prec_{{\mathcal L}(T)\otimes \mathbb M_{n_t}(\mathbb C)}^{\rm s}\Q\otimes \mathbb M_{n_t}(\mathbb C)$. Thus, $\R q'\prec_{\M}\Q \otimes \mathbb M_{n_t}(\mathbb C)$, which implies that $\Theta(\P )r \prec_{\mathscr N}\Q \otimes \mathbb M_{n_t}(\mathbb C) $. Since $\mathbb D_{n_t} (\mathbb C)\subseteq \mathbb M_{n_t}(\mathbb C)$ has finite index we further have $\Theta(\P )r \prec_{\mathscr N}\Q \otimes \mathbb D_{n_t}(\mathbb C) $. Finally, since this holds for all projections $0\neq r\in \mathscr N_{p\mathscr N p}(\Theta (\mathcal P))' \cap p\mathscr N p$, using \cite[Lemma 2.4 (b)]{DHI19} we get the desired claim. $\hfill\blacksquare$

\vskip 0.08in 
Since $\Theta(\mathcal M)\subseteq p\mathscr N p$ is a finite index inclusion of II$_1$ factors, Proposition \ref{finiteindexnorm} implies that $\Theta(\mathcal P)\subseteq \mathscr R$ admits a finite Pimsner-Popa basis and also $\mathscr R$ is a type I von Neumann algebra. Thus for every projection $0\neq r\in \mathscr R'\cap p\mathscr N p= \mathscr Z(\mathscr R)$ there is a subprojection $0\neq r_0\leqslant r$ we have that $\Theta (\mathcal P) r_0 \subseteq \mathscr R r_0$ has finite Pimsner-Popa basis and  hence $\mathscr R r_0 \prec \Theta (\mathcal P)$. Combining this with Claim \ref{theta(P)''} we get that $\sR r\prec_{\sN}\widetilde{\mathcal Q}$ and hence

\begin{equation}\label{theta(P)}
\sR\prec_{\sN}^{\text{s}}\widetilde{\mathcal Q}.
\end{equation}

Since $\widetilde{\mathcal Q}\subset \sN$ is a Cartan subalgebra, by combining \eqref{theta(P)} with \cite[Lemma 3.7]{CIOS23a} (see also \cite{Ioa11}), we get that after replacing $\Theta$ with $\text{ad}(w_0)\circ\Theta$, for some $w_0\in\sU(p\sN p)$, we may assume that 
\begin{equation}\label{theta(PP)}
\Theta(\P)\subset\widetilde{\mathcal Q} p\subset\sR.
\end{equation}

If $b\in B$, then $\Theta(u_{\widehat{b}})\in p\sN p$ normalizes $\Theta(\P)$ and thus normalizes $\sR$ as $\sR=\Theta(\P)'\cap p\sN p$. Denote $\beta_b'=\text{ad}(\Theta(u_{\widehat{b}}))\in\text{Aut}(\sR)$. Then $\beta'=(\beta_b')_{b\in B}$ defines an action of $B$ on $\sR$ which leaves $\Theta(\P)$ invariant.
The restriction of $\beta'$ to $\Theta(\P)$ is free since it is isomorphic to the conjugation action of $ B$ on $\mathcal{L}(A^{(I)})$, which is free by the hypothesis condition. Thus,
\cite[Lemma 3.8]{CIOS23a} yields an action $\beta=(\beta_b)_{b\in B}$ of $B$ on  $\sR$ satisfying 
\begin{enumerate}
\item for every $b\in B$ we have $\beta_b=\beta_b'\circ\text{ad}(\omega_b)=\text{ad}(\Theta(u_{\widehat{b}})\omega_b)$, for some $\omega_b\in\sU(\sR)$, and
\item $\widetilde{\mathcal Q} p$ is $\beta( B)$-invariant and the restriction of $\beta$ to $\widetilde{\mathcal Q} p$ is free.
\end{enumerate}

Our next goal is to apply \cite[Theorem 4.1]{CIOS23a}.
Consider the action $\alpha=(\alpha_d)_{d\in D}$ of $D$ on $\Q$ given by $\alpha_d=\Ad(v_{\widehat{d}})$ for $d\in D$.
Let $(\hat{X},\hat{\mu})$ be the dual of $C$ with its Haar measure. Let $(X,\mu)=(\hat{X}^{J},\hat{\mu}^{J})$ and $(\widetilde X,\widetilde\mu)=(X\times\mathbb Z/n_t\mathbb Z,\mu\times \mu_{n_t})$, where $\mu_{n_t}$ denotes the counting measure on $\mathbb Z/{n_t}\mathbb Z$.  Identify $\Q=\text{L}^{\infty}(X)$ and $\widetilde{\mathcal Q}=\text{L}^{\infty}(\widetilde X)$.
Denote still by $\alpha$ the corresponding measure preserving action $D\curvearrowright^{\alpha} (X,\mu)$ and let  $D\times\mathbb Z/n_t\mathbb Z\curvearrowright^{\widetilde\alpha}(\widetilde X,\widetilde\mu)$ be the action given by $(g,a)\cdot (x,m)=(g\cdot x,a+m)$. Let $X_0\subset\widetilde X$ be a measurable set such that $p=\textbf{ 1}_{X_0}$.
Since $\widetilde{\mathcal Q} p=\text{L}^{\infty}(X_0)$ is $\beta(B)$-invariant, we get a measure preserving action $B\curvearrowright^\beta (X_0,\widetilde\mu_{|X_0})$.

Note that $\widetilde{\Q}\subset \sN$ is a Cartan subalgebra and the p.m.p. equivalence relation associated to the inclusion $\widetilde{\Q}\subset \sN$ \cite{FM77}
is equal to $\sR(D\times\mathbb Z/n_t\mathbb Z\curvearrowright^{\widetilde\alpha}\widetilde X)$.
Since the restriction of $\beta$ to $\widetilde{\mathcal Q} p$ is implemented by unitaries in $p\sN p$, we deduce that
\begin{equation}
\text{$\beta( B)\cdot \mathfrak{m}\subset\widetilde\alpha(D\times\mathbb Z/n_t\mathbb  Z)\cdot \mathfrak{m}$, for almost every $\mathfrak{m}\in X_0$.}
\end{equation}

Since $B$ has property (T) and $\alpha$ is built over $D\curvearrowright J$ as a consequence of \cite[Lemma 3.4]{CIOS24}, applying \cite[Lemma 3.10]{CIOS24}, we find a partition $X_0=\sqcup_{i=1}^{n_0}X_i$ into non-null measurable sets, for some  ${n_0}\in\mathbb N\cup\{\infty\}$,  and a finite index subgroup $ B_i< B$ such that $X_i$ is $\beta( B_i)$-invariant and
 the restriction of $\beta_{| B_i}$ to $X_i$ is weakly mixing, for all $i$.

Next we will prove the following 

\begin{claim}\label{awayfromstab/centr}
There exists a sequence $(b_n)\subset B_i$  such that we have \begin{equation}\begin{split} &\lim_{n\ra \infty}\widetilde\mu(\{\mathfrak{m}\in X_0\mid \beta_{b_n}(\mathfrak{m})\in \widetilde\alpha(\widetilde{d}_1(\text{Stab}_D(j)\times\mathbb Z/n_t\mathbb Z)\widetilde{d}_2)(\mathfrak{m})\})=0\text{, and
}
\\ & \lim_{n\ra \infty}\widetilde\mu(\{\mathfrak{m}\in X_0\mid \beta_{b_n}(\mathfrak{m})\in \widetilde\alpha(\widetilde{d}_1(\text{C}_D(d)\times\mathbb Z/n_t\mathbb Z)\widetilde{d}_2)(\mathfrak{m})\})= 0,\end{split} \end{equation} 
for every $\widetilde{d}_1,\widetilde{d}_2\in D\times\mathbb Z/n_t\mathbb Z$, $j\in J$ and $d\in D$.
\end{claim}

\noindent \emph{Proof of Claim \ref{awayfromstab/centr}.} Let $H_0=\kappa_H^{-1}(Stab_D(j))$ and $H_1=\kappa_H^{-1}(C_D(d))$. Since the action $D\ca J$ has infinite orbits and $D$ is an ICC nonparabolic subgroup of a relative hyperbolic group it follows that $H_0,H_1<H$ are infinite index subgroups. 
Since $B_i \leqslant B$ is a finite index inclusion of groups then so is $\kappa^{-1}_G(B_i)\leqslant G$; hence $\Theta(\L(\kappa^{-1}_G(B_i)))\subseteq \Theta(\M)$ is finite index. Since   $\Theta(\mathcal M) \subset p\mathscr N p$ is finite index, and $(\Theta(u_g))_{g\in \kappa_G^{-1}(B_i)}$ is a group of unitaries generating $\Theta(\L(\kappa_G^{-1}(B_i)))$, Proposition \ref{infindexnonint} and Part 2.\ in Theorem \ref{corner} further imply
existence of a sequence $(g_n)\subset \kappa_G^{-1}(B_i)$ such that for every $x,y\in\mathscr N$ we have
\begin{equation}\label{weak}\text{$\lim_{n\ra \infty}\|\mathbb E_{\mathcal {L}(H_0)\,\otimes\,\mathbb M_{n_t}(\mathbb C)}(x\Theta(u_{g_n})y)\|_2= 0$ and $\lim_{n\ra \infty}\|\mathbb E_{\mathcal {L}(H_1)\,\otimes\,\mathbb M_{n_t}(\mathbb C)}(x\Theta(u_{g_n})y)\|_2= 0.$}\end{equation}

We will show that $b_n=\kappa_G(g_n)\in B_i$ satisfy the assertion of the claim.
Since $g_n^{-1}\widehat{b_n}\in A^{(I)}$ and $\omega_{b_n}\in\sU(\mathscr R)$, we get that $\Theta(u_{\widehat{b_n}})\omega_{b_n}\in\Theta(u_{g_n}) \sU(\mathscr R)$. Thus, one can find $x_1,\ldots ,  x_l$ such that  $\Theta(u_{\widehat{b_n}})\omega_{b_n}\in\sum_{i=1}^l\Theta(u_{g_n})(\widetilde{\mathcal Q} p)_1x_i$, for every $n\in\mathbb N$. As $\widetilde{\mathcal Q}$ is regular in $\sN$ and contained in $\mathcal{L}(H_0)\,\otimes\,\mathbb M_{n_t}(\mathbb C)$ and $\mathcal{L}(H_1)\,\otimes\,\mathbb M_{n_t}(\mathbb C)$,  \eqref{weak} implies that for every $x,y\in\mathscr N$
\begin{equation}\label{v_b_n}\text{$\lim_{n\ra\infty}\|\mathbb E_{\mathcal{L}(H_0)\,\overline{\otimes}\,\mathbb M_{n_t}(\mathbb C)}(x\Theta(u_{\widehat{b_n}})\omega_{b_n}y)\|_2= 0$ and $\lim_{n\ra\infty}\|\mathbb E_{\mathcal{L}(H_1)\,\otimes\,\mathbb M_{n_t}(\mathbb C)}(x\Theta(u_{\widehat{b_n}})\omega_{b_n}y)\|_2= 0$.}\end{equation}

On the other hand, we have that $\beta_{b_n}=\text{ad}(\Theta(u_{\widehat{b_n}})\omega_{b_n})$ and $\alpha_h=\text{ad}(v_{\widehat{h}})$, for every $h\in D$. Using these facts one can show that
$\widetilde\mu(\{\mathfrak{m}\in X_0\mid  \beta_{b_n}(\mathfrak{m})\in\widetilde{\alpha}(d_1 Stab_D(j)d_2\times\mathbb Z/n_t\mathbb Z)(\mathfrak{m})\})=\|\mathbb E_{\mathcal {L}(H_0)\,\otimes\,\mathbb M_{n_t}(\mathbb C)}((u_{\widehat{ d}_1}^*\otimes 1)\Theta(u_{\widehat{b_n}})\omega_{b_n}(u_{\widehat{d}_2}^*\otimes 1))\|_2^2$.
Thus, \eqref{v_b_n} proves the first assertion of Claim \ref{awayfromstab/centr}. The second assertion follows similarly. $\hfill\blacksquare$

\vskip 0.08in

Altogether, the prior relations show that all assumptions in \cite[Theorem 4.1]{CIOS23a} are satisfied and thus by the conclusion of this result we can find an injective group homomorphism $\overline{\gamma}_i: B_i\rightarrow D$ and $\varphi_i\in [\sR(D\times\mathbb Z/n_t\mathbb Z\curvearrowright^{\widetilde\alpha}\widetilde X)]$ such that $\varphi_i(X_i)=X\times\{i\}\equiv X$ and $\varphi_i\circ{\beta_b}_{|X_i}=\alpha_{\overline{\gamma}_i(b)}\circ{\varphi_i}_{|X_i}$, for all $b\in B_i$. In particular,  $\widetilde\mu(X_i)=1$. Thus,  $t=\widetilde\mu(X_0)=\sum_{i=1}^{n_0}\widetilde\mu(X_i)=n_0\in\mathbb N$. Since ${n_t}$ is the smallest integer with ${n_t}\geq t$, we get that
  ${n_t}=t=n_0$ and $p=1_{\mathscr N}$.

For $1\leq i\leq t$, let $p_i=\textbf {1}_{X_i}\in\widetilde\Q $ and $U_i\in\sN_{\sN}(\widetilde{\mathcal Q})$ such that $U_iyU_i^*=y\circ\varphi_i^{-1}$, for every $y\in\widetilde{\mathcal Q}$. Then $U_ip_iU_i^*=1\otimes e_i\in \mathcal{N}\otimes \mathbb{M}_t(\mathbb{C})$ where $e_i\in  \mathbb{M}_t(\mathbb{C})$ is the matrix that has $1$ for the $(i,i)$-entry and $0$ for the other entries. Since $\beta_b=\text{ad}(\Theta(u_{\widehat{b}})\omega_b)$ we find $(\zeta_{i,b})_{b\in B_i}\subset\sU({\Q})$ with
\begin{equation}\label{th}
\text{$U_i\Theta(u_{\widehat{b}})\omega_bp_iU_i^*=\zeta_{i,b}v_{\widehat{\overline{\gamma}_i(b)}}\otimes e_i$, for every $b\in B_i$.}
\end{equation}

Under the previous notations, we now establish  the following. 
\begin{claim}\label{finindex2} The subgroup $\overline{\gamma}_i(B_i)\leqslant D$ has finite index for all $1\leq i\leq t$.
    
\end{claim}

\noindent \emph{Proof of Claim \ref{finindex2}.} First, notice the relation \eqref{th} implies that $\mathcal{V}:=U_i\widetilde{\mathcal Q}p_iU_i^*\vee\{ U_i \Theta (u_{\hat b})\omega_b p_i U_i^*\,|\,b \in B_i\}''= \mathcal L(\kappa_H^{-1}(\overline{\gamma}_i (B_i))) \otimes e_i $. Observe that $\mathcal Q \otimes e_i=U_i\widetilde{\mathcal Q}p_iU_i^*\subseteq \mathcal V$ is a MASA and $U_i \Theta (u_{\hat b})\omega_b p_i U_i^*\in\mathscr{N}_{U_i\mathscr N p_iU_i^*}(U_i\widetilde{\mathcal Q}p_iU_i^*)$ for all $b\in B_i$.

Recall that $\Theta(\P)\subseteq \Theta(\M)\subseteq p\sN p$, $\Theta(\M)\subseteq p\sN p$ has a finite index, and $\Theta(\P)\subseteq\Theta(\M)$ is a Cartan subalgebra. From Part 1 in Proposition \ref{finiteindexnorm} there is a projection $z\in \sZ (\Theta(\P)'\cap p\sN p) =\mathscr Z(\mathscr R)\subseteq \widetilde \Q$ with $p_i z\neq 0$ such that$\Theta (\P)z \subseteq \mathscr R z $ admits a finite Pimsner-Popa basis. By \eqref{theta(PP)}, we have $[p_i,\Theta(\P)]=0$ and $p_i\in\sR$. Therefore, 
$\Theta (\P)p_iz \subseteq p_i\mathscr R p_i z $ admits a finite Pimsner-Popa basis, $\xi_1,\ldots, \xi_{l_1}\in p_i\mathscr R p_iz$.
Now fix $x\in \P$ and $b\in B_i$. Then one can see that 
 \begin{equation*}\begin{split}&
   p_i z\Theta (u_{\hat b } x) p_iz=p_i z\Theta (u_{\hat b }) \Theta(x) p_iz=p_i z \Theta (u_{\hat b }) \omega_{\hat b} \omega_{\hat b}^* \Theta(x) p_iz\\&=z\Theta (u_{\hat b }) \omega_{\hat b} p_i \omega_{\hat b}^* \Theta(x) p_i z= z\Theta (u_{\hat b }) \omega_{\hat b} p_i\Theta(x) \omega_{\hat b}^* p_iz\\&= \sum_{j=1}^{l_1} z\Theta (u_{\hat b }) \omega_{\hat b} p_i\Theta(x) \mathbb E_{\Theta(\mathcal P)p_iz}(\omega_{\hat b}^*  p_iz\xi_j^*)\xi_j\in \sum^{l_1}_{j=1}z (U_i^* \mathcal V U_i) \xi_j.\end{split}\end{equation*}

\noindent Denoting by $\mathcal S :=\Theta (\mathcal L(\kappa_G^{-1}(B_i)))$ and using basic $\|\cdot\|_2$-approximations, the prior formula further imply \begin{equation}\label{bimodule1} L^2(p_i z\mathcal S p_iz)\subseteq \overline{ \sum^{l_1}_{j=1}z(U_i^* \mathcal V U_i)\xi_j}^{\|\cdot\|_2}.
\end{equation}

\noindent Since $B_i\leq B$ has finite index by Proposition \ref{findex} then so is $p_i z\mathcal S p_i z\subseteq p_i z\mathscr N p_i z$. Using Proposition \ref{findex} there is a Pimsner-Popa basis  
$\xi'_1,\ldots,\xi'_{l_2} \in p_i z\mathscr N p_iz$ for this inclusion. Then \eqref{bimodule1} further implies that  
 \begin{equation}L^2(p_i z \mathscr N p_i z)\subseteq \overline{ \sum^{l_1}_{j_1=1}\sum^{l_2}_{j_2=1}z(U_i^* \mathcal V U_i)\xi_{j_1}\xi'_{j_2}}^{\|\cdot\|_2}.\end{equation}

Thus the $ z U^*_i \mathcal V U_i p_iz$-bimodule $L^2(p_iz \mathscr N p_iz)$ is finitely generated, of length $s$, as a left $z U^*_i \mathcal V U_i z$-module. Using \cite[Lemma 3.5]{FGS11} one can find a sequence of projections $(z_n)_n \subset (zU^*_i\mathcal V U_i z)'\cap zp_i \mathscr N z p_i$, SOT-convergent to $p_i z$  such that for every $n\in\mathbb N$ there exist $x_{1, n}, \ldots, x_{s,n}\in zp_i \mathscr N zp_i$ that are  $z U_i^*\mathcal V U_i z$-orthogonal and satisfy \begin{equation}\label{fgenbimodule}z_n y z_n = \sum_{j=1}^s \mathbb E_{ zU_i^*\mathcal V U_i z}( y x^*_{j,n})x_{j,n}\text{ for all }y\in p_i z\mathscr N p_i z.\end{equation} 

 \noindent Since $\|z_n-p_iz\|_2\ra 0$, for $n$ large enough, \eqref{fgenbimodule} implies that for every $y\in \mathscr U( p_i z\mathscr N p_i z)$ we have $0<\|p_i z\|^2/2=\sum^s_{j=1}\|\mathbb E_{ z U_i^* \mathcal V U_i z}(y x^*_{j,n})\|^2_2$. By Popa's intertwining techniques, this implies $ \mathscr N \prec_{\mathscr N}  z U_i^*\mathcal V U_i z$. Thus $\mathscr N \prec_{\mathscr N} \mathcal V= \mathcal L(\kappa_H^{-1}(\overline{\gamma}_i(B_i)))\otimes e_i$. Proposition \ref{infindexnonint} further entails that $\overline{\gamma}_i(B_i)\leq D$ has finite index.  $\hfill\blacksquare$
\vskip 0.08in
From here on our proof follows verbatim the proof of \cite[Theorem 5.1]{CIOS24}, and we include it here only for reader's convenience.

\vskip 0.08in
After replacing $ B_i$ by $ B_0=\cap_{i=1}^t B_i$ we may assume that $ B_i= B_0$, for all $1\leq i\leq t$.
We will prove the conclusion for $K=\kappa_G^{-1}( B_0)$. 
Let $\M_0=\mathcal{L}(K)$.
To prove the conclusion, it suffices to find a projection $p_0\in\Theta(\M_0)'\cap\sN$ with $(\tr\otimes\text{Tr})(p_0)=t_1\in\{1,\ldots,t\}$, homomorphisms $\gamma:K\rightarrow H$, $\rho:K\rightarrow\sU_{t_1}(\mathbb C)$, $\widetilde{U}\in\sU(\sN)$ such that $\gamma$ is injective, $\widetilde{U}p_0\widetilde{U}^*=1\otimes (\sum_{i=1}^{t_1}e_i)$ and $\widetilde{U}\Theta(u_g)p_0\widetilde{U}^*=v_{\gamma(g)}\otimes\rho(g)$, for all $g\in K$.
 Indeed, once we have this assertion, the conclusion will follow by a maximality argument. This concludes the first part of the proof.

The rest of the proof, which is divided between four claims (Claims \ref{P0}-\ref{Ji}), is devoted to proving the last assertion. We begin with the following:

\begin{claim}\label{P0}
There are $1\leq {t_1}\leq t$ and a homomorphism $\overline{\gamma}: B_0\rightarrow D$ such that, after renumbering, we have $p_0=\sum_{i=1}^{t_1}p_i\in\Theta(\M_0)'\cap \sN$ and can take $\overline{\gamma}_i=\overline{\gamma}$, for every $1\leq i\leq {t_1}$.
\end{claim}

\emph{Proof of Claim \ref{P0}.}
We start by showing that if $1\leq i,j\leq t$ and $p_i\sR p_j\not=\{0\}$, then $\overline{\gamma}_i$ and $\overline{\gamma}_j$ are conjugate. Let $x\in (\sR)_1$ with $p_ixp_j\not=0$. Then $x_b=\text{ad}(\Theta(u_{\widehat{b}})\omega_b)(p_ix)=\beta_b(p_ix)\in (\sR)_1$ and $(\Theta(u_{\widehat{b}})\omega_bp_i)xp_j=x_b(\Theta(u_{\widehat{b}})\omega_bp_j)$, for every $b\in B_0$. 
Using \eqref{th} we derive that \begin{equation}\label{zzeta}\text{$\big(U_i^*(\zeta_{i,b}v_{\widehat{\overline{\gamma}_i(b)}}\otimes e_i)U_i \big)xp_j=x_b \big(U_j^*(\zeta_{j,b}v_{\widehat{\overline{\gamma}_j(b)}}\otimes e_j)U_j\big)$, for every $b\in B_0$.}\end{equation}

For every subset $F\subset D$, let $P_{F}$ be the orthogonal projection from ${L}^2(\sN)$ onto the $\|\cdot\|_2$-closed linear span of $\{v_h\otimes x\mid h\in\kappa_H^{-1}(F),x\in\mathbb M_t(\mathbb C)\}$. Since $(\zeta_{i,b})_{b\in B_0}, (\zeta_{j,b})_{b\in B_0}\subset\sU(\Q)$ and $(x_b)_{b\in B_0}\subset (\sR)_1$, then using basic $\|\cdot\|_2$-approximations and also
 \eqref{theta(P)} in combination with \cite[Lemma 2.5]{Vae13} (for b)), we  can find finite $F\subset D$ so that for every $b\in B$ we have 
 \begin{equation*}\begin{split}& \text{a)} \qquad \|\big(U_i^*(\zeta_{i,b}v_{\widehat{\overline{\gamma}_i(b)}}\otimes e_i)U_i \big)xp_j-P_{F\overline{\gamma}_i(b)F}(\big(U_i^*(\zeta_{i,b}v_{\widehat{\overline{\gamma}_i(b)}}\otimes e_i)U_i \big)xp_j)\|_2<\frac{\|p_ixp_j\|_2}{2}, \; \text{ and  }\\ &
 \text{b)} \qquad \|x_b \big(U_j^*(\zeta_{j,b}v_{\widehat{\overline{\gamma}_j(b)}}\otimes e_j)U_j\big)-P_{F\overline{\gamma}_j(b)F}(x_b \big(U_j^*(\zeta_{j,b}v_{\widehat{\overline{\gamma}_j(b)}}\otimes e_j)U_j\big))\|_2<\frac{\|p_ixp_j\|_2}{2}.\end{split}\end{equation*}
 
 Combining a)-b) with \eqref{zzeta}, we derive that $F\overline{\gamma}_i(b)F\cap F\overline{\gamma}_j(b)F\not=\emptyset$, for every $b\in B_0$.
Since $\overline{\gamma}_i$ is injective and $ B_0< B$ has finite index, Claim \ref{finindex2} further implies that $\overline{\gamma}_i( B_0)\leq  D$ has finite index.
Since for any $d\in D\setminus\{1\}$, $\text{C}_{D}(d)< D$ has infinite index  then $\{\overline{\gamma}_i(b)d\overline{\gamma}_i(b)^{-1}\mid b\in B_0\}$ is infinite.  By \cite[Lemma 7.1]{BV14} we find $d\in D$ such that $\overline{\gamma}_i(b)=d\overline{\gamma}_j(b)d^{-1}$, for every $b\in B_0$. This proves our assertion that $\overline{\gamma}_i$ and $\overline{\gamma}_j$ are conjugate.

Let $\overline{\gamma}=\overline{\gamma}_1: B_0\rightarrow D$.
After renumbering, we may assume that $\overline{\gamma}_1,\ldots,\overline{\gamma}_{t_1}$ are conjugate to $\overline{\gamma}$ and $\overline{\gamma}_{{t_1}+1},\ldots,\overline{\gamma}_t$ are not conjugate to $\overline{\gamma}$, for some $1\leq {t_1}\leq t$. The previous paragraph implies that  $p_i\sR p_j=\{0\}$, for every $1\leq i\leq {t_1}$ and ${t_1}+1\leq j\leq t$. Thus, $p_0=\sum_{i=1}^{t_1}p_i$ belongs to the center of $\sR$. As $p_0$ commutes with $\Theta(u_{\widehat{b}})\omega_b$ and $\omega_b\in\sU(\sR)$, then $p_0$ commutes with $\Theta(u_{\widehat{b}})$, for every $b\in B_0$. 
Since $p_i\in \widetilde\Q p\subset \sR=\Theta(\P)'\cap\sN$, for every $1\leq i\leq t$, $p_0\in \Theta(\P)'\cap\sN$. Thus, $p_0$ also commutes with $\Theta(\P)$.
Since $\P$ and $(u_{\widehat{b}})_{b\in B_0}$ generate $\M_0$, we get that $p_0\in\Theta(\M_0)'\cap \sN$. Moreover, if $1\leq i\leq {t_1}$,  there is $d_i\in D$ such that $\overline{\gamma}_i(h)=d_i\overline{\gamma}(h)d_i^{-1}$, for every $h\in B_0$. After replacing $\varphi_i$ by $\alpha_{d_i^{-1}}\circ\varphi_i$ we may assume that $\overline{\gamma}_i=\overline{\gamma}$, for every $1\leq i\leq {t_1}$. $\hfill\blacksquare$

Let $U=\sum_{i=1}^{t_1}U_ip_i$ and $e=\sum_{i=1}^{t_1}e_i$. Then $U$ is a partial isometry, $UU^*=1\otimes e, U^*U=p_0$ and $U\widetilde{\mathcal Q} p_0U^*=\widetilde{\mathcal Q}(1\otimes e)$. If we let $\zeta_b=\sum_{i=1}^{t_1}\zeta_{i,b}\otimes e_i\in\sU(\widetilde{\mathcal Q}(1\otimes e))$, then \eqref{th} gives that
\begin{equation}\label{thetap_0}\text{$U\Theta(u_{\widehat{b}})\omega_bp_0U^*=\zeta_b(v_{\widehat{\overline{\gamma}(b)}}\otimes e)$, for every $b\in B_0$.}\end{equation}

Identify $\sN_1:=(1\otimes e)\sN(1\otimes e)$ and $\Q_1:=\widetilde{\mathcal Q}(1\otimes e)$ with $\Nn\overline{\otimes}\mathbb M_{t_1}(\mathbb C)$ and $\mathcal Q\overline{\otimes}\mathbb D_{t_1}(\mathbb C)$, respectively.
Consider the unital $*$-homomorphism $\Theta_1:\M_0\rightarrow \sN_1$ given by $\Theta_1(x)=U\Theta(x)p_0U^*$, for every $x\in \M_0$, and let $\sR_1=\Theta_1(\P)'\cap \sN_1$.  Letting $w_b=U\omega_bp_0U^*\in\sU(\sR_1)$, then \eqref{thetap_0} rewrites as
\begin{equation}\label{thetap}\text{$\Theta_1(u_{\widehat{b}})w_b=\zeta_b(v_{\widehat{\overline{\gamma}(b)}}\otimes 1)$, for every $b\in B_0$. }\end{equation}

By \eqref{theta(PP)}, $\Q_1\subset\sR_1$.
Since $(\Theta_1(u_{\widehat{b}})w_b)_{b\in B}$ normalizes $\sR_1$ and $(\zeta_b)_{b\in B_0}\subset\sU(\Q_1)$, \eqref{thetap} implies that $(v_{\widehat{\overline{\gamma}(b)}}\otimes 1)_{b\in B_0}$ normalizes $\sR_1$.  Thus,  $\eta_b=\zeta_b\text{ad}(v_{\widehat{\overline{\gamma}(b)}}\otimes 1)(w_b^*)\in\sU(\sR_1)$ and
\begin{equation}\label{etah}
\text{$\Theta_1(u_{\widehat{b}})=\eta_b(v_{\widehat{\overline{\gamma}(b)}}\otimes 1)$, for every $b\in B_0$.}
\end{equation}

\begin{claim}\label{T} $\sR_1={\mathcal Q}\overline{\otimes}\T$, for a von Neumann subalgebra $\T\subset\mathbb M_{t_1}(\mathbb C)$.
\end{claim}
\noindent\emph{Proof of Claim \ref{T}.}
First, we show that $\sR_1\subset {\mathcal Q}\overline{\otimes}\mathbb M_{t_1}(\mathbb C)$.
To this end, since $\overline{\gamma}(B_0)\leqslant D $ has finite index and $\text{C}_D(d)$ has infinite index in $D$ then for every $d\in D\setminus\{1\}$, there is a sequence $(b_n)\subset B_0$ such that for every $d_1,d_2\in D$ and $d'\in D\setminus\{1\}$ we have $d_1\text{ad}(\overline{\gamma}(b_n))(d')d_2\not=1$, for every $m$ large enough. 

Next we claim that $\|\mathbb{E}_{\Q}(x_1\text{ad}(v_{\widehat{\overline{\gamma}(b_n)}})(b)x_2)\|_2\rightarrow 0$, for every $x_1,x_2\in\Nn$ and $y\in\Nn\ominus\Q$. We only have to check this for $x_1=v_{h_1},x_2=v_{h_2},y=v_k$, where $h_1,h_2\in H,k\in H\setminus C^{(I)}$. In this case $\kappa_H(k)\not=1$, thus $\kappa_H(h_1\text{ad}(\widehat{\overline{\gamma}(b_n)})(k)h_2))=\kappa_H(h_1)\text{ad}(\overline{\gamma}(b_n))(\kappa_H(k))\kappa_H(h_2)\not=1$ and therefore $\mathbb{E}_{\Q}(x_1\text{ad}(v_{\widehat{\overline{\gamma}(b_n)}})(y)x_2)=0$, for every $m$ large enough. 

The previous paragraph further implies that \begin{equation}\label{conv'}\|\mathbb{E}_{\Q_1}(x_1\text{ad}(v_{\widehat{\overline{\gamma}(b_n)}}\otimes 1)(y)x_2)\|_2\rightarrow 0, \text{ for every } x_1,x_2\in\sN_1\text{ and } y\in\sN_1\ominus (\Q\overline{\otimes}\mathbb M_{t_1}(\mathbb C)).\end{equation}  Let $\sZ(\sR_1)$ be the center of $\sR_1$. Since $\Q_1\subset\sN_1$ is a MASA we have $\sZ(\sR_1) \subseteq \Q_1\subseteq \sR_1$. Moreover, since $\Theta_1(\mathcal M_0)\subseteq \mathscr N_1$ has finite index then so does $\Theta_1(\P)\subseteq \mathscr R_1$. Thus, using Proposition \ref{finiteindexnorm}   
 $ \mathcal Q_1 \subseteq \mathscr R_1$ admits a finite Pimser-Popa basis.  Combining this with \eqref{conv'}, basic approximations show that 

\begin{equation}\label{ortho}\text{$\|\mathbb{E}_{\sR_1}(\text{ad}(v_{\widehat{\overline{\gamma}(b_n)}}\otimes 1)(y))\|_2\rightarrow 0$, for every $y\in\sN_1\ominus (\Q\overline{\otimes}\mathbb M_{t_1}(\mathbb C))$.} \end{equation}

To see that $\sR_1\subset {\mathcal Q}\overline{\otimes}\mathbb M_{t_1}(\mathbb C)$, let $y_0\in \sR_1$. Put $y_1=\mathbb{E}_{ {\mathcal Q}\overline{\otimes}\mathbb M_{t_1}(\mathbb C)}(y_0)$ and $y_2=y_0-y_1$.
Since $v_{\widehat{\overline{\gamma}(b_n)}}\otimes 1$ normalizes $\sR_1$, $\text{ad}(v_{\widehat{\overline{\gamma}(b_n)}}\otimes 1)(y_0)\in\sR_1$ and thus $\|\mathbb{E}_{\sR_1}(\text{ad}(v_{\widehat{\overline{\gamma}(b_n)}}\otimes 1)(y_0))\|_2=\|y_0\|_2$, for every $m$. 
On the other hand, $\|\mathbb{E}_{\sR_1}(\text{ad}(v_{\widehat{\overline{\gamma}(b_n)}}\otimes 1)(y_2))\|_2\rightarrow 0$ by \eqref{ortho}. Thus, we get that $\|\mathbb{E}_{\sR_1}(\text{ad}(v_{\widehat{\overline{\gamma}(b_n)}}\otimes 1)(y_1))\|_2\rightarrow \|y_0\|_2$. Since $\|\mathbb{E}_{\sR_1}(\text{ad}(v_{\widehat{\overline{\gamma}(b_n)}}\otimes 1)(y_1))\|_2\leq \|y_1\|_2\leq \|y_0\|_2$, we conclude that $\|y_1\|_2=\|y_0\|_2$. Hence, $y_0=y_1\in {\mathcal Q}\overline{\otimes}\mathbb M_{t_1}(\mathbb C).$ This proves that  $\sR_1\subset {\mathcal Q}\overline{\otimes}\mathbb M_{t_1}(\mathbb C)$.

Therefore, we have that  $\Q_1={\mathcal Q}\overline{\otimes}\mathbb D_{t_1}(\mathbb C)\subseteq\sR_1\subseteq {\mathcal Q}\overline{\otimes}\mathbb M_{t_1}(\mathbb C).$
Since $\Q=\text{L}^{\infty}(X,\mu)$, we can disintegrate $\sR_1=\int_X^{\oplus}\T_\mathfrak{m}\;\text{d}\mu(\mathfrak{m})$, where $(\T_\mathfrak{m})_{\mathfrak{m}\in X}$ is a measurable field of von Neumann subalgebras of $\mathbb M_{t_1}(\mathbb C)$ containing $\mathbb D_{t_1}(\mathbb C)$.
We denote $$\text{$\widetilde{\alpha}_b=\text{ad}(v_{\widehat{\overline{\gamma}(b)}})\in\text{Aut}(\Q)$, for $b\in B_0$.}$$ Then $\widetilde{\alpha}=(\widetilde{\alpha}_b)_{b\in B_0}$ defines an action of $ B_0$ on $\Q$.  Since $(v_{\widehat{\overline{\gamma}(b)}}\otimes 1)_{b\in B_0}$ normalizes $\sR_1$, we have $\T_{\widetilde{\alpha}_b(\mathfrak{m})}=\T_\mathfrak{m}$, for every $b\in B_0$ and almost every $\mathfrak{m}\in X$.  \cite[Lemma 3.4]{CIOS23a} implies that $\widetilde{\alpha}$ is built over the action $B_0\curvearrowright J$ given by $b\cdot j=\overline{\gamma}(b)j$.
Since $\text{Stab}_D(j)<D$ has infinite index, $\overline{\gamma}$ is injective and $\overline{\gamma}(B_0)<D$ has finite index, the action $B_0\curvearrowright J$ has infinite orbits. This implies that $\widetilde{\alpha}$ is weakly mixing, hence ergodic. Thus, one can find a von Neumann subalgebra $\T\subset\mathbb M_{t_1}(\mathbb C)$ such that $\T_\mathfrak{m}=\T$, for almost every $\mathfrak{m}\in X$, which proves the claim.  $\hfill\blacksquare$
\vskip 0.08in
Next, let $\sU=\sU(\T)/\sU(\sZ(\T))$, where $\sZ(\T)$ is the center of $\T$, and $\mathfrak{q}:\sU(\T)\rightarrow\sU$ be the quotient homomorphism. We continue with the following claim:

\begin{claim}\label{nu}
There are maps $\overline{\rho}: B_0\rightarrow\sU(\T)$ and $\nu: B_0\rightarrow\sU(\sZ(\sR_1))$ such that the map $ B_0\ni b\mapsto \mathfrak{q}(\overline{\rho}_b)\in\sU$ is a homomorphism and after replacing  $\Theta_1$ by $\emph{ad}(\sigma)\circ\Theta_1$, for some $\sigma\in\sU(\sR_1)$, we have that
$\Theta_1(u_{\widehat{b}})=\nu_b(v_{\widehat{\overline{\gamma}(b)}}\otimes \overline{\rho}_b)$, for every $b\in B_0$.
\end{claim}

\noindent \emph{Proof of Claim \ref{nu}.}
Note that $(\text{ad}(\Theta_1(u_{\widehat{b}})))_{b\in B_0}$ and $(\widetilde{\alpha}_b\otimes\text{Id}_{\T})_{b\in B_0}$ define actions of $ B_0$ on $\sR_1=\Q \otimes \T$. Combining this observation with \eqref{etah} gives that
\begin{equation}\label{cocycle}
\text{$\eta_{b_1b_2}^*\eta_{b_1}(\widetilde{\alpha}_{b_1}\otimes\text{Id})(\eta_{b_2})\in\sZ(\sR_1)=\Q \otimes\sZ(\T)$, for every $b_1,b_2\in B_0$.}
\end{equation}

Viewing every $\eta\in\sU(\sR_1)$ as a measurable function $\eta:X\rightarrow\sU(\T)$, \eqref{cocycle} rewrites as $\eta_{b_1b_2}(\mathfrak{m})^*\eta_{b_1}(\mathfrak{m})\eta_{b_2}(\widetilde{\alpha}_{b_1}^{-1}\mathfrak{m})\in\sU(\sZ(\T))$, for every $b_1,b_2\in B_0$ and almost every $\mathfrak{m}\in X$. 
 Then $\psi_0 : B_0\times X\rightarrow\sU$ given by $\psi_0(b,\mathfrak{m})=\mathfrak{q}(\eta_b(\mathfrak{m}))$ is a 1-cocycle for $\widetilde{\alpha}$.
Since $\sU(\sZ(\T))$ is a closed central subgroup of the compact Polish group $\sU(\T)$, $\sU$ is a compact Polish group with respect to the quotient topology.
In particular, $\sU$ is a $\sU_{\text{fin}}$ group (see \cite[Lemma 2.7]{Pop06}). Since $B_0$ has property (T), $\widetilde{\alpha}$ is built over $B_0\curvearrowright J$ and $B_0\curvearrowright J$ has infinite orbits, \cite[Theorem 3.5]{CIOS24} (see also \cite[Theorem 3.6]{CIOS23a}) implies that
$\psi_0$ is cohomologous to a homomorphism $\psi: B_0\rightarrow\sU$. 

Let $\sigma:X\rightarrow\sU(\T)$ be a measurable map satisfying $\mathfrak{q}(\sigma(\mathfrak{m}))\psi_0(b,\mathfrak{m})\mathfrak{q}(\sigma(\widetilde{\alpha}_{b^{-1}}(\mathfrak{m})))^{-1}=\psi_b$, for every $b\in B_0$ and almost every $\mathfrak{m}\in X$.
 Let $\overline{\rho}: B_0\rightarrow\sU(\T)$ be  such that $\mathfrak{q}(\overline{\rho}_b)=\psi_b$. Thus, we find a measurable map $\nu_b:X\rightarrow\sU(\sZ(\T))$ such that $\sigma(\mathfrak{m})\eta_b(\mathfrak{m})\sigma(\widetilde{\alpha}_{b^{-1}}(\mathfrak{m}))^{-1}=\nu_b(\mathfrak{m})\overline{\rho}_b$, for every $b\in B_0$ and almost every $\mathfrak{m}\in X$.
Equivalently, $\sigma\in\sU(\sR_1)$ and $\nu: B_0\rightarrow\sU(\sZ(\sR_1))$ satisfy $\sigma\eta_b(\widetilde{\alpha}_b\otimes\text{Id})(\sigma)^*=\nu_b(1\otimes\overline{\rho}_b)$, for every $b\in B_0$. This implies the claim.$\hfill\blacksquare$ 
\vskip 0.08in 
We are now ready to finish the proof of the main assertion.
Let $g\in K=\kappa_G^{-1}( B_0)$  and put $b=\kappa_G(g)\in B_0$. Then $a=g\widehat{b}^{-1}\in A$ and $\Theta_1(u_a)\in \sZ(\sR_1)$. Denoting $V_g=\Theta_1(u_a)\nu_b$ and using Claim \ref{nu}, we get that $V_g\in\sU(\sZ(\sR_1))$ and
\begin{equation}\label{w_g}
\text{$\Theta_1(u_g)=V_g(v_{\widehat{\overline{\gamma}(\kappa_G(g))}}\otimes \overline{\rho}_{\kappa_G(g)})$, for every $g\in K$.}
\end{equation}

Our final claim is the following.

\begin{claim}\label{Ji}
There are maps $\lambda:K\rightarrow C^{(I)}$, $z:K\rightarrow\sZ(\T)$ and $w\in \sU(\sZ(\sR_1))$ such that
$V_g=(v_{\lambda_g}\otimes z_g)w^*(\widetilde{\alpha}_{\kappa_G(g)}\otimes\emph{Id})(w)$, for every $g\in K$.
\end{claim}

\noindent \emph{Proof of Claim \ref{Ji}.}
We define $\Lambda:K\times K\rightarrow C^{(I)}$  by $\Lambda_{g,h}=\widehat{\overline{\gamma}(\kappa_G(g))}\widehat{\overline{\gamma}(\kappa_G(h))}\widehat{\overline{\gamma}(\kappa_G(gh))}^*$ and  $Z:K\times K\rightarrow\sU(\sZ(\T))$ by $Z_{g,h}=\overline{\rho}_{\kappa_G(gh)}\overline{\rho}_{\kappa_G(h)}^*\overline{\rho}_{\kappa_G(g)}^*$.
Since $\Theta_1(u_{g})\Theta_1(u_{h})=\Theta_1(u_{gh})$, \eqref{w_g} gives
\begin{equation}\label{2c}
\text{$v_{\Lambda_{g,h}}\otimes 1=\big(V_{gh}V_g^*(\widetilde{\alpha}_{\kappa_G(g)}\otimes\text{Id})(V_h)^*\big)\big(1\otimes Z_{g,h}\big)$, for every $g,h\in K$.}
\end{equation}

Note that the map $K\times K\ni(g,h)\mapsto v_{\Lambda_{g,h}}\in \sU(\Q)$ is a $2$-cocycle for the action $K\curvearrowright^{\widetilde{\alpha}\circ\kappa_G}\Q$ and the map $K\times K\ni (g,h)\mapsto Z_{g,h}\in\sU(\sZ(\T))$ is a $2$-cocycle for the trivial action. Since $\widetilde{\alpha}$ is built over $B_0\curvearrowright J$, $\widetilde{\alpha}\circ\kappa_G$ is built over the action $K\curvearrowright J$ given by $k\cdot j=\kappa_G(k)\cdot j$, for every $k\in K$ and $j\in J$. Since the action $K\curvearrowright J$ has infinite orbits, $\widetilde{\alpha}\circ\kappa_G$ is weakly mixing.
If $n\in\mathbb N$, then the action $(\widetilde{\alpha}\circ\kappa_G)^{\otimes n}$ is built over the diagonal product action $K\curvearrowright I^n$, which has infinite orbits.
 Since $K$ has property (T), \cite[Theorem 3.5]{CIOS24} implies that $(\widetilde{\alpha}\circ\kappa_G)^{\otimes n}$ is $\sU_{\text{fin}}$-cocycle superrigid.
Altogether, we deduce that $\widetilde{\alpha}\circ\kappa_G$ satisfies the hypothesis of \cite[Theorem 4.1]{CIOS24}.

Let $\{f_1,\ldots,f_l\}$ be the minimal projections of $\sZ(\T)$. Let $1\leq i\leq l$.
Since $Z_{g,h}f_i\in\mathbb Tf_i$, for every $g,h\in K$, by using \eqref{2c} and applying \cite[Theorem 4.1]{CIOS24} there are maps $\lambda:K\rightarrow C^{(I)}$ and $c^i:K\rightarrow\mathbb T$ such that $v_{\Lambda_{g,h}}=v_{\lambda_{gh}}v_{\lambda_g}^*\widetilde{\alpha}_{\kappa_G(g)}(v_{\lambda_h})^*$ and $Z_{g,h}f_i=\overline{c_{gh}^i}{c_g^i}{c_h^i}f_i$, for all $g,h\in K$.
Define $c:K\rightarrow\sU(\sZ(\T))$ by letting $c_g=\sum_{i=1}^lc_g^if_i$, for every $g\in K$. Then $Z_{g,h}=c_{gh}^*c_gc_h$, for all $g,h\in K$.
Using \eqref{2c}, we get that the map $K\ni g\mapsto (v_{\lambda_g}^*\otimes c_g^*)V_g\in\sU(\sZ(\sR_1))$ is a $1$-cocycle for the action $K\curvearrowright^{\widetilde{\alpha}\circ\kappa_G\otimes\text{Id}}\sZ(\sR_1)=\Q\overline{\otimes}\sZ(\T)$. Since $\widetilde{\alpha}\circ\kappa_G$ is $\sU_{\text{fin}}$-cocycle superrigid, we deduce the existence of $w\in  \sU(\sZ(\sR_1))$ such that the claim holds.
$\hfill\blacksquare$
\vskip 0.08in 
Finally, \eqref{w_g} and Claim \ref{Ji} imply that $\Theta_1(u_g)=w^*(v_{\lambda_g\widehat{\overline{\gamma}(\kappa_G(g))}}\otimes c_g\overline{\rho}_{\kappa_G(g)})w$, for every
 $g\in K$. Then $\gamma:K\rightarrow H$ and $\rho:K\rightarrow\sU(\T)\subset\sU_{t_1}(\mathbb C)$ given by $\gamma(g)=\lambda_g\widehat{\overline{\gamma}(\kappa_G(g))}$ and $\rho(g)=c_g\overline{\rho}_{\kappa_G(g)}$ must be homomorphisms. 
Then $\Theta_1(u_g)=w^*(v_{\gamma(g)}\otimes\rho(g))w$, for every $g\in K$. By construction $A\subset K$ and $\gamma(A)\subset C^{(I)}$.
 Moreover, if $g\in\ker(\gamma)$, then we have that $\Theta_1(u_g)=w^*(1\otimes\rho(g))w\in w^*(1\otimes\mathbb M_{t_1}(\mathbb C))w$. This implies that $\ker(\gamma)$ must be finite. Since $G$ and thus $K$ are ICC it follows that $\gamma$ is injective. This finishes the proof of the main assertion.
\end{proof}

\vskip 0.1in
 Theorem \ref{symmetries} leads to a complete description of all virtual $*$-isomorphisms $\Theta:\L(G)\rightarrow\L(H)^t$. To explain this, we assume the setting of Theorem \ref{symmetries} and introduce some terminology from \cite[Section 2]{PV22}. 
 
 Let $K<G$ be a finite index subgroup. 
  If $\gamma:K\rightarrow H$  and $\rho:K\rightarrow\mathscr U_s(\mathbb C)$ are homomorphisms, for some $s\in\mathbb N$, we denote by $\pi_{\gamma,\rho}:K\rightarrow \sU(\mathcal{L}(H)\overline{\otimes}\mathbb M_s(\mathbb C))$ the homomorphism given by $\pi_{\gamma,\rho}(g)=v_{\gamma(g)}\otimes\rho(g)$ for  $g\in K$. If $\pi:K\rightarrow\sU(\mathcal S)$ is a homomorphism, where $\mathcal S$ is a tracial von Neumann algebra, we denote by
 $\text{Ind}_K^G(\pi):G\rightarrow\sU(\mathcal S\overline{\otimes}\mathbb M_{[G:K]}(\mathbb C))$ the induced homomorphism. Specifically, let $\chi:G/K\rightarrow G$ be a map such that $\chi(gK)\in gK$, for every $g\in G$, and define $c:G\times G/K\rightarrow K$ by letting $c(g,hK)=\chi(ghK)^{-1}g\chi(hK)\in K$, for every $g,h\in G$.
Identifying $\mathbb M_{[G:K]}(\mathbb C)=\mathbb B(\ell^2(G/K))$, we define $\text{Ind}_K^G(\pi)(g)(\xi\otimes\textbf{1}_{hK})=\pi(c(g,hK))\xi\otimes\textbf{1}_{ghK}$.
By \cite[Defintion 2.1]{PV22}, a homomorphism $G\rightarrow \sU(\mathcal{L}(H)^t)$ is called \textit{standard} if it is unitarily conjugate to a direct sum of homomorphisms of the form $\text{Ind}_K^G(\pi_{\gamma,\rho})$ induced from finite index subgroups $K<G$.
Theorem \ref{symmetries} implies that the restriction of any $*$-homomorphism $\Theta:\mathcal{L}(G)\rightarrow\mathcal{L}(H)^t$  to $G=\{u_g\}_{g\in G}$ is standard. 
\vskip 0.05in
We end this section with a result which is very similar in nature with \cite[Theorem 5.6]{CIOS24}. Also our proof largely follows the same argument verbatim, and we include it here only for reader's convenience.

\begin{thm}\label{explicit} Let $G,H$ be groups as in Theorem \ref{symmetries}.
Let $\Theta:\mathcal{L}(G)\rightarrow\mathcal{L}(H)^t$ be a virtual $*$-isomorphism, for some $t>0$. Then $t\in\mathbb N$ and we can find ${m}\in\mathbb N$ and for every $1\leq i\leq {m}$, a finite index subgroup $K_i\leqslant G$, an injective homomorphism $\gamma_i:K_i\rightarrow H$ with finite index image and a unitary representation $\rho_i:K_i\rightarrow\mathscr U_{s_i}(\mathbb C)$, for some $s_i\in\mathbb N$,  and a unitary $w\in\mathcal{L}(H)^t=\mathcal{L}(H)\overline{\otimes}\mathbb M_t(\mathbb C)$  such that $\sum_{i=1}^{m}[G:K_i]s_i=t$ and
$$\text{$w\Theta(u_g)w^*=\emph{diag}(\emph{Ind}_{K_1}^G(\pi_{\gamma_1,\rho_1})(g),\ldots,\emph{Ind}_{K_{m}}^G(\pi_{\gamma_{m},\rho_{m}})(g))$, for every $g\in G$.}$$

\noindent Moreover, we have the following formula for the index of the image of $\Theta$

$$
[\mathcal{L}(H)^t:\Theta(\mathcal{L}(G))]=t\sum_{i=1}^{m} s_i[H:\gamma_i(K_i)].
$$

\end{thm}

\begin{proof}
By Theorem \ref{symmetries},  there are $t_1,\ldots, t_{m'}\in\mathbb N$ with $t_1+\cdots+t_{m'}=t$, for some $m'\in\mathbb N$, a finite index subgroup $K\leqslant G$, for every $1\leq i\leq m'$, an injective homomorphism $\gamma_i':K\rightarrow H$ with a finite index image, and a unitary representation $\rho_i':K\rightarrow\sU_{t_i}(\mathbb C)$, and a unitary $w\in \mathcal{L}(H)^t=\mathcal{L}(H)\overline{\otimes}\mathbb M_t(\mathbb C)$ such that 
$$\text{$ w\Theta(u_g)w^*=\text{diag}(v_{\gamma_1'(g)}\otimes\rho_1'(g),
\ldots, v_{\gamma_{m'}'(g)}\otimes\rho_{m'}'(g))
$, for every $g\in K$.}$$

\noindent Passing to a finite index subgroup we can assume without any loss of generality that $K\lhd G$ is a normal subgroup.
After decomposing each $\rho_i'$, into a direct sum of irreducible representations, we may assume  $\rho_i'$ is irreducible, for every $1\leq i\leq {m'}$.
Further, we can find $m_1,\ldots,m_l,n_1,\ldots,n_l\in\mathbb N$, for some $l\in\mathbb N$, with $m_1n_1+\cdots+m_ln_l=t$, for every $1\leq i\leq l$, an injective homomorphism $\gamma_i^0:K\rightarrow H$ with a finite index image,  an irreducible representation $\rho_i^0:K\rightarrow\sU_{m_i}(\mathbb C)$, and a unitary $w_0\in \L(H)^t$ such that for every $1\leq i<j\leq l$,
$\gamma_i^0$ is not conjugate to $\gamma_j^0$ or $\rho_i^0$ is not unitarily conjugate to $\rho_j^0$,
and after replacing $\Theta$ by $\text{ad}(w_0)\circ\Theta$ we have
\begin{equation}\label{thetaonK}
\text{$\Theta(u_g)=\text{diag}(v_{\gamma_1^0(g)}\otimes\rho_1^0(g)\otimes I_{n_1},
\ldots, v_{\gamma_l^0(g)}\otimes\rho_l^0(g)\otimes I_{n_l})
$, for every $g\in K$.}
\end{equation}

\noindent Here, two homomorphisms $\gamma,\gamma':K\rightarrow H$ are conjugate if $\gamma={\rm ad}(h)\circ \gamma'$ for some $h\in H$.  We denote by $I_d\in\mathbb M_d(\mathbb C)$ the identity matrix, for every $d\in\mathbb N$, and consider the natural unital embedding $\bigoplus_{i=1}^l\big(\mathcal{L}(H)\otimes\mathbb M_{m_i}(\mathbb C)\otimes\mathbb M_{n_i}(\mathbb C))\subset \mathcal{L}(H)\otimes\mathbb M_t(\mathbb C).$ 

We continue with the following claim:

\begin{claim}\label{intertwiners}
Let $\gamma,\gamma':K\rightarrow H$ be injective homomorphisms with finite index images
and $\rho:K\rightarrow\sU_{a}(\mathbb C), \rho':K\rightarrow\sU_{a'}(\mathbb C)$ be irreducible representations, for some $a,a'\in\mathbb N$.
 Assume there  is $0\neq x\in \mathcal{L}(H)\otimes\mathbb M_{a,a'}(\mathbb C)$ such that   $(v_{\gamma(g)}\otimes\rho(g))x=x(v_{\gamma'(g)}\otimes\rho'(g))$, for every $g\in K$.
Then $a=a'$ and there exist $h\in H$ and $U\in\sU_a(\mathbb C)$ such that $\gamma(g)h=h\gamma'(g)$ and
$\rho(g)U=U\rho'(g)$, for every $g\in K$. Moreover, $x=c(v_h\otimes U)$ for some $c\in\mathbb C$.
\end{claim}

\noindent\emph{Proof of Claim \ref{intertwiners}}
Write $x=\sum_{h\in H}v_h\otimes x_h$, where $x_h\in\mathbb M_{a,a'}(\mathbb C)$ and $\sum_{h\in H}\|x_h\|_2^2=\|x\|_2^2<\infty$. Then 
\begin{equation}\label{equivariance}
    \text{$x_{\gamma(g)h\gamma'(g)^{-1}}=\rho(g)x_h\rho'(g)^*$, for every $g\in K, h\in H.$}
\end{equation}
  Let $\varepsilon>0$ such that $F=\{h\in H\mid \|x_h\|_2>\varepsilon\}$ is nonempty. Then \eqref{equivariance} implies that $F$ is a finite set such that $\gamma(g)F\gamma'(g)^{-1}=F$, for every $g\in K$.   Thus, $\gamma(g)FF^{-1}\gamma(g)^{-1}=FF^{-1}$, for every $g\in K$. Assume by contradiction $FF^{-1}\neq\{1\}$. Thus, as $F$ is finite there exist $1\neq h\in FF^{-1}$ and a finite index subgroup $G_0\leqslant G$ such that $\gamma (G_0)\leqslant C_H(h)$. As $\gamma (G)\leqslant H$ has finite index it follows that $C_H(h)\leqslant H$ has finite index. However, this contradicts that $H$ is ICC.
  Thus $FF^{-1}=\{1\}$ and hence $F$ consists of a single element. 
  Since this holds for every small enough $\varepsilon>0$, we conclude that $\{h\in H\mid x_h\not=0\}$ has a single element. Thus, $x=v_h\otimes x_h$, for some $h\in H$.
  Hence, we have that $v_{\gamma(g)h}\otimes\rho(g)x_h=v_{h\gamma'(g)}\otimes x_h\rho'(g)$, which implies that $\gamma(g)h=h\gamma'(g)$ and $\rho(g)x_h=x_h\rho'(g)$, for every $g\in K$.
  Since $\rho,\rho'$ are irreducible, the latter relation implies that $a=a'$ and $x_h=c U$, for some nonzero $c\in \mathbb C$ and $U\in\sU_a(\mathbb C)$. This proves our assertion. $\hfill\blacksquare$
\vskip 0.08in

Let $[l]=\{1,\cdots, l\}$. For every $i\in [l]$, let $f_i=1\otimes I_{m_i}\otimes I_{n_i}$.
Combining \eqref{thetaonK} with Claim \ref{intertwiners} we get that $\Theta(\mathcal{L}(K))'\cap \mathcal{L}(H)^t=\bigoplus_{i=1}^l(1\otimes I_{m_i}\otimes\mathbb M_{n_i}(\mathbb C))$, and thus $\sZ(\Theta(\mathcal{L}(K))'\cap \mathcal{L}(H)^t)=\bigoplus_{i=1}^l\mathbb Cf_i$.
Since $K<G$ is normal, $\Theta(u_g)$ normalizes $\Theta(\mathcal{L}(K))$ for all $g\in G$. 
Hence, there exists an action $G\curvearrowright [l]$ such that $\Theta(u_g)f_i\Theta(u_g)^*=f_{g\cdot i}$, for every $g\in G$ and $i\in [l]$. Let $J\subset [l]$ be a set which intersects every $G$-orbit exactly once. 

Next, fix $i\in J$ and denote $K_i=\{g\in G\mid g\cdot i=i\}$. 
Then \eqref{thetaonK} implies that $K<K_i$.
Let $h\in K_i$. If $g\in K$, then since $K<G$ is normal, $hgh^{-1}\in K,$ and \eqref{thetaonK} gives that \begin{equation}\label{conjug}\text{$\Theta(u_{hgh^{-1}})=\text{diag}(v_{\gamma_1^0(hgh^{-1})}\otimes\rho_1^0(hgh^{-1})\otimes I_{n_1},\ldots, v_{\gamma_l^0(hgh^{-1})}\otimes\rho_l^0(hgh^{-1})\otimes I_{n_l})$. }
\end{equation}

Since $\Theta(u_{hgh^{-1}})\Theta(u_h)=\Theta(u_h)\Theta(u_g)$, we get that $\Theta(u_{hgh^{-1}})(\Theta(u_h)f_i)=(\Theta(u_h)f_i)\Theta(u_g)$.
By combining \eqref{thetaonK} and \eqref{conjug} we conclude that for every $g\in K$ we have
\begin{equation}
\text{$(v_{\gamma_i^0(hgh^{-1})}\otimes\rho_i^0(hgh^{-1})\otimes I_{n_i})(\Theta(u_h)f_i)=(\Theta(u_h)f_i)(v_{\gamma_i^0(g)}\otimes\rho_i^0(g)\otimes I_{n_i})$.}
\end{equation}

\noindent By applying the moreover part of Claim \ref{intertwiners}, we get that $\Theta(u_h)f_i=v_{\gamma_i(h)}\otimes \rho_i(h)$, for some $\gamma_i(h)\in H$ and $\rho_i(h)\in\sU_{m_in_i}(\mathbb C)$.
Then $\gamma_i:K_i\rightarrow H$ and $\rho_i:K_i\rightarrow\sU_{m_in_i}(\mathbb C)$ must be homomorphisms such that $\gamma_i(g)=\gamma_i^0(g)$ and $\rho_i(g)=\rho_i^0(g)\otimes I_{n_i}$, for every $g\in K$, and $s_i=m_in_i$ satisfies $\sum_{i\in J}[G:K_i]s_i=t$. Thus, in the notation introduced before this proof, we have that $\Theta(u_h)f_i=\pi_{\gamma_i,\rho_i}(h)$, for every $h\in K_i$.

Let $e_i=\sum_{j\in G\cdot i}f_{j}\in \sZ(\Theta(\mathcal{L}(K))'\cap \mathcal{L}(H)^t)$. Then $e_i\in\Theta(\mathcal{L}(G))'\cap\mathcal{L}(H)^t$, hence $e_i\in \sZ(\Theta(\mathcal{L}(G))'\cap \mathcal{L}(H)^t)$. 
Since $e_i=\sum_{g\in G/K_i}\Theta(u_g)f_i\Theta(u_g)^*$ and the projections $\{\Theta(u_g)f_i\Theta(u_g)^*\mid g\in G/K_i\}$ are pairwise orthogonal, the homomorphism $G\ni g\mapsto \Theta(u_g)e_i\in\sU(e_i\mathcal{L}(H)^te_i)$ is unitarily conjugate to the induced homomorphism $\text{Ind}_{K_i}^G(\pi_{\gamma_i,\rho_i})$. 

Since $\Theta(u_g)=\sum_{i\in J}\Theta(u_g)e_i$, for every $g\in G$, the conclusion follows.

\vskip 0.1in

Now we prove the moreover part.  
Let $\{h_l\}_{l}$ be the coset representatives of $\gamma_i(K_i)\le H$. Then $\{\sqrt{s_i} v_{h_l}\otimes e_{j_1,j_2}\}_{j_1,j_2,l}$ is a Pimsner-Popa basis for $f_i\Theta(\mathcal{L}(K_i))f_i\subset f_i\mathcal{L}(H)^tf_i$.
Thus, 
$$
[f_i\mathcal{L}(H)^tf_i:f_i\Theta(\mathcal{L}(K_i))f_i]=s_i^2[H:\gamma_i(K_i)].
$$
Using the above and the local index formula, i.e.\ item 7.\ in Propositon \ref{findex}, we have $$
    [e_i\mathcal{L}(H)^te_i:e_i\Theta(\mathcal{L}(K_i))e_i]=[G:K_i]\frac{s_i^2[H:\gamma_i(K_i)]}{1/[G:K_i]}=s_i^2[G:K_i]^2[H:\gamma_i(K_i)].
    $$
    Since $[\mathcal{L}(G):\mathcal{L}(K_i)]=[G:K_i]$ we get
    $$
    [e_i\mathcal{L}(H)^te_i:e_i\Theta(\mathcal{L}(G))e_i]=\frac{[e_i\mathcal{L}(H)^te_i:e_i\Theta(\mathcal{L}(K_i))e_i]}{[G:K_i]}=s_i^2[G:K_i][H:\gamma_i(K_i)].
    $$
    Again, by Proposition \ref{findex} 7.\ and the above, we have
    \[
        [\mathcal{L}(H)^t:\Theta(\mathcal{L}(G))]=\sum_{i=1}^{m} \frac{[e_i\mathcal{L}(H)^te_i:e_i\Theta(\mathcal{L}(G))e_i]}{s_i[G:K_i]/t}=\sum_{i=1}^{m} \frac{s_i^2[G:K_i][H:\gamma_i (K_i)]}{s_i[G:K_i]/t}=t\sum_{i=1}^{m} s_i[H:\gamma_i(K_i)].\qedhere
    \]
\end{proof}

\section{Jones index set for II$_1$ factors associated with property (T) wreath-like product groups}
Let $\mathcal M$ be a II$_1$ factor. The \textit{Jones index set} $\mathscr{I}(\mathcal M)$ the collection of the Jones indices of all finite index subfactors of $\mathcal M$. In this section, we show that there are a continuum family of property (T) II$_1$ factors, whose Jones index set contains all the positive integers. To achieve this, we will use in an essential way the following result  \cite[Corollary 2.12]{CIOS23b}.
\begin{thm}[\cite{CIOS23b}]\label{out thm}
    Let $\Gamma$ be any countable group. Then one can find a continuum of ICC property (T) groups $(G_j)_{j\in J}$ such that the corresponding II$_1$ factors $(\mathcal L(G_j))_{j\in J}$ are pairwise non-stably isomorphic and satisfy $\Out(\mathcal L(G_j))\cong \Gamma$ and $\mathcal{F}(\mathcal L(G_j))=\{1\}$ for every $j\in J$.
\end{thm}

We give a brief summary of the proof. First, they found a continuum family of ICC property (T) countable groups $(G_j)_{j\in J}$ with $\Out(G_j)\cong \Gamma$ for every $j\in J$. The groups $(G_j)_{j\in J}$ satisfy the conditions in Theorem \ref{symmetries}. i.e. $G_j$ are in $\WR(A_j,B\curvearrowright I)$ where $A_j$ are nontrivial abelian groups, $B$ is an nonparabolic ICC subgroup of a finitely generated group that are hyperbolic relative to a finite family of residually finite groups, and $B\curvearrowright I$ is a faithful action. Then they introduced a homomorphism $\Psi^{(j)}:\Aut(G_j)\rightarrow \Aut(\L(G_j))$ defined by 
\begin{equation}
\Psi^{(j)}(\varphi)(u_g)=u_{\varphi(g)}\label{isomorphism psi}
\end{equation}
for each $j\in J$. Next, they showed that for the above family of groups, $\Psi^{(j)}$ induces an isomorphism $\overline{\Psi}^{(j)}:\Out(G_j)\rightarrow \Out(\L(G_j))$ defined by
$$
\overline{\Psi}^{(j)}(\varphi\Inn(G_j))=\Psi^{(j)}(\varphi)\Inn(\L(G_j)).
$$
The isomorphisms $\overline{\Psi}^{(j)}$ are key ingredients for the construction of property (T) factors whose Jones index sets are exactly the positive integers.

\begin{thm}\label{Jones index set}
    There is a continuum  of ICC property (T) groups $(G_j)_{j\in J}$, such that their II$_1$ factors $\mathcal L(G_j)$ are pairwise not stably isomorphic, and $\mathscr{I}(\mathcal L(G_j))=\mathbb{N}$. Moreover, all integer indices are realized as those of the irreducible subfactors of $\mathcal L(G_j)$.
\end{thm}
\begin{proof}
    Let $\Gamma$ be a countable group such that for all positive integer $n$ there is a subgroup $\Gamma_0\le \Gamma$ with $|\Gamma_0|=n$. For instance, one can choose $\Gamma=\bigoplus_{n\in\mathbb{N}}\mathbb{Z}/n\mathbb{Z}$. By Theorem \ref{out thm}, we can take a continuum family of ICC property (T) groups $(G_j)_{j\in J}$ such that $(G_j)_{j\in J}$ satisfy the conditions in Theorem \ref{symmetries}, the II$_1$ factors $(\L(G_j))_{j\in J}$ are pairwise not stably isomorphic and $\Out(\L(G_j))\cong \Gamma$ for every $j\in J$. 

    First, we show that $\mathscr{I}(\L(G_j))\subset \mathbb{N}$. Set $\M=\L(G_j)$ and let $\mathcal{N}\subset \M$ be a finite index subfactor. Consider the iterated basic construction $\mathcal{N}\subset\M\subset \M_1\subset \M_2$. Then $\M_2\cong \M^{[\M:\mathcal{N}]}$ Proposition \ref{amplification} and this isomorphism induces an embedding $\M\hookrightarrow\M^{[\M:\mathcal{N}]}$. By Theorem \ref{symmetries}, we have $[\M:\mathcal{N}]\in\mathbb{N}$.
    
    Next, we show that $\mathbb{N}\subset \mathscr{I}(\L(G_j))$ for all $j\in J$. 
    For each $n\in\mathbb{N}$, choose a subgroup $\Gamma_0\le \Gamma$ with $|\Gamma_0|=n$. Recall that the homomorphism $\Psi^{(j)}$ of (\ref{isomorphism psi}) induces an isomorphism $\overline{\Psi}^{(j)}:\Out(G_j)\rightarrow\Out(\L(G_j))$. Therefore, applying Proposition \ref{cocycle construction} gives an irreducible subfactor $\mathcal N\subset \L(G_j)$ such that $[\L(G_j):\mathcal N]=n$. Hence, $n\in \mathscr{I}(\L(G_j))$.    
\end{proof}

As another consequence of Theorem \ref{out thm}, we obtain property (T) group factors whose Jones index sets are proper subsets of the positive integers:
\begin{cor}\label{proper set}
    There is a continuum  of ICC property (T) groups $(G_j)_{j\in J}$, such that their II$_1$ factors $\mathcal L(G_j)$ are pairwise not stably isomorphic, and $\{2k-1 \,:\,k\in \mathbb N\}\subseteq \mathscr{I}(\mathcal L(G_j))\subseteq\mathbb{N}\setminus \{2\}$. Furthermore, there is such a continuum family $(G_j)_{j\in J}$ with $ \{1\}\cup \{6k\pm 1 \,:\, k \in \mathbb N\}\subseteq \mathscr{I}(\mathcal L(G_j))\subseteq \mathbb N \setminus\{2,3\}$.
\end{cor}
\begin{proof}
    Let $\Gamma$ be a countable group such that $\Gamma$ does not have an order $2$ element and have subgroups of arbitrary odd orders. For instance, we can choose $\Gamma=\bigoplus_{2\nmid n}\mathbb{Z}/n\mathbb{Z}$. By Proposition \ref{cocycle construction} we have $\{2k-1 \,:\,k\in \mathbb N\}\subseteq \mathscr{I}(\mathcal L(G_j))$. Let $(G_j)_j$ be as in Theorem \ref{out thm} with $\Out(\L(G_j))\cong \Gamma$ for all $j$.
    Any index $2$ subfactor is of the form $\M\subset \M\rtimes \mathbb{Z}/2\mathbb{Z}$ \cite{Gol59, Jon83}. Thus, by the choice of $\Gamma$, we have $2\notin \mathscr{I}(\L(G_j))$. 

    Now, let $\Gamma$ be a countable group without an element of order $2$ and $3$ that has subgroup with order $6k\pm1$ for all $k\in\mathbb{N}$. For instance, we can choose $\Gamma=\bigoplus_{2,3\nmid n}\mathbb{Z}/n\mathbb{Z}$. By Proposition \ref{cocycle construction} we have $\{6k\pm1 \,:\,k\in \mathbb N\}\subseteq \mathscr{I}(\mathcal L(G_j))$. The principal graph of any index $3$ subfactor is either $A_5$ or $D_4$ \cite{GdlHJ89, Ocn88, Kaw95}. Any subfactor with $A_5$-principal graph is of the form $\M\rtimes S_2 \subset \M\rtimes S_3$ where $S_n$ is the symmetric group of degree $n$ \cite{Izu92}. If $\L(G_j)\cong \M\rtimes S_2$ then it has an index $2$ subfactor isomorphic to $\M\subset \M\rtimes S_2$, which contradicts the previous argument. Any subfactor with $D_4$-principal graph is of the form $\M\subset \M\rtimes \mathbb{Z}/3\mathbb{Z}$ \cite{OK90}. Since $\Gamma$ does not have an order $3$ element $\L(G_j)$ cannot be such an $\M$.
\end{proof}

\begin{appendices}

\section{The basic construction of ICC group subfactors}

In this section, we characterize the basic construction for the inclusion of von Neumann algebras of finite index ICC subgroups. This is well-known to experts, but we could not find a reference.
Similar constructions can be found in \cite{OK90, KY92,JS97}.
\subsection{Amplificational description of the basic construction}
The following proposition can be shown simply by iterating \cite[Proposition 1.5]{PP86}.
\begin{prop}[{\cite[Proposition 1.5]{PP86}}]\label{amplification}
    Let $\Nn\subset\M$ be an inclusion of $\rm{II}_1$ factors with $t=[\M:\Nn]<\infty$, $\{x_i\}_{i\in I}\subset \mathscr M$ be a Pimsner-Popa basis where $|I|\leq \lfloor t\rfloor +1$. Consider the Jones tower
    $$
    \Nn\subset \M\stackrel{e_1}{\subset} \M_1\stackrel{e_2}{\subset}\M_2\stackrel{e_3}{\subset}\cdots
    $$
    obtained by iterating the basic construction.
    Then 
    \begin{enumerate}
        \item The set $\{x_i^{(k)}\}_{i\in I}$ of elements defined by $x_i^{(k)}=t^{k/2}e_ke_{k-1}\cdots e_1x_i$ forms a Pimsner-Popa basis for the inclusion $\M_{k-1}\subset\M_k$.
    \end{enumerate}
    For each pair $\mathbf{i},\mathbf{j}\in I^k$ let $\Theta_{\mathbf{ij}}^{(k)}:\M_{2k}\rightarrow \M$ be defined by
    \begin{align*}
        \Theta_{\mathbf{ij}}^{(k)}(y)&= \mathbb{E}_{\M}\left(x_{i_1}^{(2)}\mathbb{E}_{\M_{2}}\left(x_{i_2}^{(4)}\cdots \mathbb{E}_{\M_{2k-2}}\left(x_{i_k}^{(2k)}y\left(x_{j_k}^{(2k)}\right)^*\right)\cdots\left(x_{j_2}^{(4)}\right)^*\right)\left(x_{j_1}^{(2)}\right)^*\right)
    \end{align*}
    and $\Theta^{(k)}:\M_{2k}\rightarrow \M^{t^k}$ by $\left(\Theta^{(k)}(y)\right)_{\mathbf{ij}}=\Theta^{(k)}_{\mathbf{ij}}(y)$.
    \begin{enumerate}
        \setcounter{enumi}{1}
        \item The map $\Theta^{(k)}$ is a $*$-isomorphism with $\Theta^{(k)}(\M_{2k-1})=\Nn^{t^k}$.
        \item For every $y\in\M$ we have 
        $$
        \Theta_{\mathbf{ij}}^{(k)}(y)= \mathbb{E}_{\Nn}\left(x_{i_1}\mathbb{E}_{\Nn}\left(x_{i_2}\cdots \mathbb{E}_{\Nn}\left(x_{i_k}y\left(x_{j_k}\right)^*\right)\cdots\left(x_{j_2}\right)^*\right)\left(x_{j_1}\right)^*\right).
        $$
    \end{enumerate}
\end{prop}

\subsection{Inclusions of von Neumann algebras of ICC groups}
Let $H\le G$ be an inclusion of ICC groups with $t=[G:H]<\infty$. In this subsection, we describe the structure of the inclusion $\L(H)\subset \L(G)$ of their von Neumann algebras. In Section \ref{ICC standard invariant}, we describe the standard invariant of $\L(H)\subset\L(G)$. In Section \ref{ICC cocycle description} we describe $\L(H)$ and $\L(G)$ as cocycle crossed product algebras. 

\subsubsection{The standard invariant}\label{ICC standard invariant}
We describe the standard invariant of $\L(H)\subset\L(G)$ in terms of the representations of subgroups of $G$. We mimicked the description given in \cite[Appendix A.4]{JS97} for the subfactors induced by the finite group actions.

Let $I=\{1,2,\cdots, t\}$ and $\{g_i\}_{i\in I}$ be a full set of the right coset representatives of $H$, i.e. $G=\coprod_{i=1}^t Hg_i$. For simplicity, we choose $g_1=1$.
Then $\{u_{g_i}\}_{i=1}^t$ is a Pimsner-Popa basis for $\L(H)\subset \L(G)$. Let $k\in\mathbb{N}$ and $\Theta_{ij}^{(k)}:\L(G)\rightarrow \L(H)$ be as in Proposition \ref{amplification}. For $\mathbf{i}\in I^k$ let $g_{\mathbf{i}}=g_{i_1}g_{i_2}\cdots g_{i_k}$.
Then for all $g\in G$
\begin{equation}
    \Theta_{\mathbf{ij}}^{(k)}(u_g)=\left\{\begin{array}{cc}
         u_{g_{\mathbf{i}}gg_{\mathbf{j}}^{-1}},& g_{i_l}g_{i_{l+1}}\cdots g_{i_k}g\in Hg_{j_l}g_{j_{l+1}}\cdots g_{j_k} \text{ for all }1\le l\le k,  \\
     0,&\text{otherwise} 
    \end{array}\right..\label{action on I^k}
\end{equation}
Note that for each $\mathbf{i}$ there is a unique $\mathbf{j}$ such that $\Theta_{\mathbf{ij}}^{(k)}(u_g)\ne 0$ and vice versa. Motivated by the above, we define the action $G\curvearrowright I^k$ by 
\begin{equation}
    g\cdot \mathbf{j} = \mathbf{i} \Leftrightarrow  g_{j_l}g_{j_{l+1}}\cdots g_{j_k}g^{-1}\in Hg_{i_l}g_{i_{l+1}}\cdots g_{i_k}\text{ for all }1\le l\le k.\label{gp action on index}
\end{equation}

We compute the higher relative commutants of $\L(H)\subset\L(G)$. Let $g_0\in G$ and $x=(x_{\mathbf{ij}})_{\mathbf{i},\mathbf{j}}\in \L(G)^{t^k}$. Then 
\begin{align}
    \left[x,\Theta^{(k)}(u_g)\right]=0
    &\ \Leftrightarrow\  x_{\mathbf{i},(g_0\cdot\mathbf{j})}\ u_{g_{(g_0\cdot\mathbf{j})}g_0g_{\mathbf{j}}^{-1}}
    =u_{g_{\mathbf{i}}g_0g_{\left(g_0^{-1}\cdot\mathbf{i}\right)}^{-1}} x_{\left(g_0^{-1}\cdot\mathbf{i}\right),\mathbf{j}}&\text{ for all } \mathbf{i},\mathbf{j}\in I^k\nonumber\\
    &\ \Leftrightarrow\  x_{(g_0\cdot\mathbf{i}),(g_0\cdot\mathbf{j})}\ u_{g_{(g_0\cdot\mathbf{j})}g_0g_{\mathbf{j}}^{-1}}
    =u_{g_{(g_0\cdot\mathbf{i})}g_0g_{\mathbf{i}}^{-1}} x_{\mathbf{i},\mathbf{j}}&\text{ for all } \mathbf{i},\mathbf{j}\in I^k\label{entry twist 1}
\end{align}
Let $K=\bigcap_{i\in I}g_i^{-1}Hg_i$. Then $K\lhd G$ and $[G:K]<\infty$. If $g_0\in K$ then $g_0\cdot \mathbf{i}=\mathbf{i}$ for all $\mathbf{i}\in I^k$ and the above condition becomes
\begin{align}
    &x_{\mathbf{i},\mathbf{j}}\ u_{g_{\mathbf{j}}g_0g_{\mathbf{j}}^{-1}}
    =u_{g_{\mathbf{i}}g_0g_{\mathbf{i}}^{-1}} x_{\mathbf{i},\mathbf{j}}&\text{ for all } \mathbf{i},\mathbf{j}\in I^k.\label{entry twist 2}
\end{align}
Suppose that $x\in \Theta(\L(K))'\cap \L(G)^{t^k}$. Let $x_{\mathbf{i},\mathbf{j}}=\sum_{g\in G}c_g u_g$, $\varepsilon>0$, and $F=\{g\in G: |c_g|>0\}$. Since $\sum_{g\in G}|c_g|^2=\|x_{\mathbf{i},\mathbf{j}}\|^2<\infty$, $F$ is finite. Moreover, from \eqref{entry twist 2}, we have that $F=(g_{\mathbf{i}}g_0g_{\mathbf{i}}^{-1})F(g_{\mathbf{j}}g_0g_{\mathbf{j}}^{-1})^{-1}$ for all $g_0\in K$. Thus, $FF^{-1}=(g_{\mathbf{i}}g_0g_{\mathbf{i}}^{-1})FF^{-1}(g_{\mathbf{i}}g_0g_{\mathbf{i}}^{-1})^{-1}$ for all $g_0\in K$. Since $G$ is ICC and $[G:K]<\infty$, we have $FF^{-1}=\{1\}$. Hence, $F$ is a singleton. Since this is true for all $\varepsilon>0$, we have $x_{\mathbf{i},\mathbf{j}}=c_{\mathbf{i},\mathbf{j}}u_g$ for some $g\in G$ and some $c_{\mathbf{i},\mathbf{j}}\in \mathbb{C}$. Suppose $c_{\mathbf{i},\mathbf{j}}\ne 0$. Plugging this back to \eqref{entry twist 2} yields $[g_{\mathbf{i}}^{-1}gg_{\mathbf{j}},g_0]=0$ for all $g_0\in K$. Again, since $G$ is ICC and $[G:K]<\infty$, $g_{\mathbf{i}}^{-1}gg_{\mathbf{j}}=1$. Therefore, 
\begin{equation}
    x_{\mathbf{i},\mathbf{j}}=c_{\mathbf{i},\mathbf{j}}u_{g_{\mathbf{i}}g_{\mathbf{j}}^{-1}}.\label{entry}
\end{equation}

Let $H\le G_0\le G$. We will describe the structures of the inclusions $\Theta^{(k)}(\L(G_0))\subset \L(G)^{t^k}$ and $\Theta^{(k)}(\L(G_0))\subset \L(H)^{t^k}$ by analyzing their relative commutants. Our main focus will be on the cases $G_0=G,H$.

First, assume that $x\in \Theta^{(k)}(\L(G_0))'\cap \L(G)^{t^k}$. Then by \eqref{entry twist 1}, we further have that 
\begin{align}
    c_{(g_0\cdot \mathbf{i}),(g_0\cdot \mathbf{j})}&=c_{\mathbf{i},\mathbf{j}}&\text{ for all } g_0\in G_0 \text{ and } \mathbf{i},\mathbf{j}\in I^k.\label{intertwining condition}
\end{align}
Thus, the minimal central projections of $\Theta^{(k)}(\L(G_0))'\cap \L(G)^{t^k}$ correspond to the irreducible representations of $G_0$ with nonzero multiplicity in the permutation representation of $G$ on $\mathbb{C}I^k$. The rank of the minimal central projections of $\Theta^{(k)}(\L(G_0))'\cap \L(G)^{t^k}$ is the multiplicity of the corresponding irreducible representation in $\mathbb{C}I^k$.

Next, we consider $\Theta^{(k)}(\L(G_0))'\cap \L(H)^{t^k}$. Notice that $\Theta^{(k)}(\L(G_0))'\cap \L(H)^{t^k}\subset \Theta^{(k)}(\L(K))'\cap \L(H)^{t^k}$. Suppose that $x\in \Theta^{(k)}(\L(K))'\cap \L(H)^{t^k}$. Then by \eqref{entry}, for all $\mathbf{i},\mathbf{j}\in I^k$ there exists $c_{\mathbf{i},\mathbf{j}}\in\mathbb{C}$ such that
\begin{align}
    x_{\mathbf{i},\mathbf{j}}&=\left\{\begin{array}{cc}
        c_{\mathbf{i},\mathbf{j}}u_{g_{\mathbf{i}}g_{\mathbf{j}}^{-1}},& g_{\mathbf{i}}g_{\mathbf{j}}^{-1}\in H, \\
     0,&\text{otherwise} 
    \end{array}\right.. \label{entry H}
\end{align}
Therefore,
$$
\Theta^{(k)}(\L(K))'\cap \L(H)^{t^k}= \bigoplus_{m\in I} \left\{x: x_{\mathbf{i},\mathbf{j}}=0 \text{ if } g_{\mathbf{i}}\notin Hg_m\text{ or } g_{\mathbf{j}}\notin Hg_m\right\}\cong \bigoplus_{m\in I}\mathbb{M}_{I^{k-1}}(\mathbb{C}).
$$
The algebra $\Theta^{(k)}(\L(G_0))'\cap \L(H)^{t^k}$ consists of elements satisfying \eqref{entry H} and \eqref{intertwining condition}. Let $J\subset I$ be a subset that contains exactly one element from each orbit of the action $G_0\curvearrowright I$. For each $i\in J$, let $K_i\le G_0$ be the stabilizer subgroup of $i$. Then 
$$
\Theta^{(k)}(\L(G_0))'\cap \L(H)^{t^k}\cong \bigoplus_{i\in J} (K_i'\cap M_{I^{k-1}}(\mathbb{C})).
$$
Thus, $\Theta^{(k)}(\L(G_0))'\cap \L(H)^{t^k}$ is isomorphic to the direct sum of endomorphism spaces of copies of $\mathbb{C}I^{k-1}$, each of which is viewed as a representation of $K_i$.
For $G_0=G$, we may choose $J=\{1\}$ and $K_1=H$. For $G_0=H$, the orbits of $G_0\curvearrowright I$ correspond to the double cosets of $H$, and $J$ can be chosen that $G=\coprod_{i\in J} Hg_i H$. Moreover, for each $i\in J$ $K_i=H\cap g_i^{-1}Hg_i$.

\begin{prop}
    Let $H\le G$ be an inclusion of ICC groups with $[G:H]<\infty$. Denote the set of isomorphism classes of the finite-dimensional irreducible representations of a group $G_0$ as $\widehat{G_0}$.
    Then the principal graphs $(\Delta, \Delta')$ of $\L(H)\subset \L(G)$ have the following description:
    \begin{enumerate}
        \item Let $\{g_i\}_{i=1}^q$ be a full set of the double coset representatives of $H$ and $K_i=H\cap g_i^{-1}Hg_i$. Let $\widetilde{\Delta}_0=\coprod_{i=1}^q \widehat{K}_i$ and $\widetilde{\Delta}_1= \widehat{H}$. Connect $\rho_0\in \widetilde{\Delta}_0$ and $\rho_1\in \widetilde{\Delta}_1$ with $m$ edges if the multiplicity of $\rho_0$ in $\rho_1$ is $m$. Then $\Delta$ is the connected component of the resulting graph $\widetilde{\Delta}$ that contains $1\in \widetilde{\Delta}_1$ with the designated vertex $1\in \widehat{H}\subset \widetilde{\Delta}_0$.
        \item Let $\widetilde{\Delta}_0'=\widehat{G}$ and $\widetilde{\Delta}_1'=\widehat{H}$. Connect $\rho_0\in \widetilde{\Delta}_0'$ and $\rho_1\in \widetilde{\Delta}_1'$ with $m$ edges if the multiplicity of $\rho_1$ in $\rho_0$ is $m$. Then $\widetilde{\Delta}$ is the connected component of $\widetilde{\Delta}'$ that has $1\in \widehat{G}=\widetilde{\Delta}_0'$ as the designated vertex. 
    \end{enumerate}
\end{prop}

\subsubsection{Cocycle action description}\label{ICC cocycle description}
Recall that $K=\bigcap_{i=1}^t g_i^{-1}Hg_i\lhd G$ where $\{g_i\}_{i=1}^t$ is a full set of the right coset representatives of $H$. Let $\{\widetilde{g}_j\}_j$ be a full set of right coset representatives of $K\subset G$ that contains a full set of right coset representatives of $K\subset H$.
Define a map $\alpha:G/K\rightarrow \Aut(\L(K))$ and a  $2$-cocycle $\omega:G/K\times G/K \rightarrow \L(K)$
\begin{align*}
    \alpha_{K\widetilde{g}_j}=\Ad(u_{\widetilde{g}_j}),&
    &\omega_{K\widetilde{g}_i,K\widetilde{g}_j}=g
    &\text{ where } \widetilde{g}_i\widetilde{g}_j=g\widetilde{g}_k \text{ for some } \widetilde{g}_k \text{ and } g\in K.
\end{align*}
Then $\L(K)\rtimes_{\alpha,\omega}G/K\rightarrow \L(G)$ defined by 
$$
u_gv_{K\widetilde{g}_j}\mapsto u_{g\widetilde{g}_j} 
$$
is a $*$-isomorphism and the restriction of this map to $\L(K)\rtimes_{\alpha,\omega}H/K$ has the image $\L(H)$.
\end{appendices}

\end{document}